\documentclass[0.2pt,reqno]{amsart}
\usepackage[utf8]{inputenc}
\usepackage{mathrsfs}
\usepackage{amsfonts}
\usepackage{amsmath,amssymb,latexsym,esint,mathrsfs}
\usepackage{verbatim,cite}
\usepackage[left=2.1cm,right=2.1cm,top=1.3cm,bottom=1.3cm]{geometry}

\usepackage{color,enumitem,graphicx}
\usepackage[colorlinks=true,urlcolor=blue, citecolor=red,linkcolor=blue,
linktocpage,pdfpagelabels,
bookmarksnumbered,bookmarksopen]{hyperref}

\usepackage[hyperpageref]{backref}
\usepackage[english]{babel}

\newtheorem{theorem}{Theorem}[section]
\newtheorem{lemma}{Lemma}[section]
\newtheorem{corollary}{Corollary}[section]
\newtheorem{proposition}{Proposition}[section]

\newtheorem{remark}{Remark}[section]
\newtheorem{definition}{Definition}[section]
\numberwithin{equation}{section} %%
\makeindex

\newcommand{\s}{\section}

\newcommand{\R}{\mathbb R}

\renewcommand{\div}{{\rm div}}

\newcommand{\wF}{\widetilde F}

\newcommand{\e}{\varepsilon}

\newcommand{\bt}{\begin{theorem}}
\newcommand{\et}{\end{theorem}}
\newcommand{\bl}{\begin{lemma}}
\newcommand{\el}{\end{lemma}}
\newcommand{\bd}{\begin{definition}}
\newcommand{\ed}{\end{definition}}
\newcommand{\bc}{\begin{corollary}}
\newcommand{\ec}{\end{corollary}}
\newcommand{\bp}{\begin{proof}}
\newcommand{\ep}{\end{proof}}
\newcommand{\bx}{\begin{example}}
\newcommand{\ex}{\end{example}}
\newcommand{\bi}{\begin{exercise}}
\newcommand{\ei}{\end{exercise}}
\newcommand{\bo}{\begin{proposition}}
\newcommand{\eo}{\end{proposition}}
\newcommand{\br}{\begin{remark}}
\newcommand{\er}{\end{remark}}
\newcommand{\be}{\begin{equation}}
\newcommand{\ee}{\end{equation}}
\newcommand{\ba}{\begin{align}}
\newcommand{\ea}{\end{align}}
\newcommand{\bn}{\begin{enumerate}}
\newcommand{\en}{\end{enumerate}}
\newcommand{\bg}{\begin{align*}}
\newcommand{\bcs}{\begin{cases}}
\newcommand{\ecs}{\end{cases}}

\newcommand{\HR}{{\mathcal H}}

\newcommand{\SR}{{\mathbb S}}

\newcommand{\WR}{{\mathcal W}}

\newcommand{\bean}{\begin{eqnarray*}}
\newcommand{\eean}{\end{eqnarray*}}

\begin{document}

\title[Sharp capillary Sobolev Inequality and Moser-Trudinger Inequality Outside Convex Domain]{ Sharp capillary Sobolev Inequality and Moser-Trudinger Inequality Outside Convex Domain}

%\author[Lu Chen, Jiali Lan ]{Lu Chen, Jiali Lan }

\author{Lu Chen}
\address[Lu Chen]{Key Laboratory of Algebraic Lie Theory and Analysis of Ministry of Education, School of Mathematics and Statistics, Beijing Institute of Technology, Beijing
100081, PR China; Tangshan Research Institute, Beijing Institute of Technology, Tangshan 063000, PR China} 
\email{chenlu5818804@163.com}

\author{Jiali Lan}
\address[Jiali Lan]{Key Laboratory of Algebraic Lie Theory and Analysis of Ministry of Education, School of Mathematics and Statistics, Beijing Institute of Technology, Beijing
100081, PR China}
\email{17636268505@163.com}

%\date{\today}
%%    \thanks will become a 1st page (unnumbered) footnote.
%%    Multiple \thanks are possible. You may use it for the
%%    indication of the corresponding author.

%\thanks{X.\ He is  supported by NSFC (11771468, 11271386)}

%\subjclass[2000]{35J62, 35J50, 35B65}
%\date{\today}

\thanks{The first author was partly supported by the  National Natural Science Foundation of China (No. 12271027) and Hebei Natural Science Foundation (No. A2025105003).}

\begin{abstract}
The theory of sharp geometric inequality in $\mathbb{R}^n$ and inside convex cone has been well-developed, much less known for sharp capillary geometric inequality outside convex domain. Recently, Fusco-Julin-Morini-Pratelli \cite{FJMP} obtained sharp capillary isoperimetric inequality and make it possible to obtain the sharp capillary geometric inequality outside convex domain. In this paper, we establish the sharp capillary Sobolev inequality and Moser-Trudinger inequality outside convex domain, which can be seen as geometric inequality on the Finsler manifold to some extent. Our method is based on constructing capillary P\'{a}lya-Szeg\"{o} rearrangement inequality outside convex domain.
Finally, we also consider the capillary Talenti-Comparison principle and Bossel-Daners inequality.

\end{abstract}

\maketitle

%\begin{center}
%\begin{minipage}{9.5cm}
%\small
%\tableofcontents
%\end{minipage}
%\end{center}

\vspace{-1cm}
 \section{Introduction and Main Results}

Given $n \geq 3$ and $1<p<n$, the classical Sobolev inequality \cite{S} in $\mathbb{R}^n$ states that for any $u\in W^{1,p}(\mathbb{R}^n)$, there holds
\begin{equation}\label{1.1}
\int_{\mathbb{R}^n}|u|^{p^{*}}dx \leq C(n, p, s)\int_{\mathbb{R}^n}|\nabla u |^{p}dx,
\end{equation}
where $p^{*} = \frac{np}{n-p}$ denotes the critical Sobolev exponent. Aubin \cite{A2} and Talenti \cite{T1} applied the technique of symmetry and rearrangement combining the Bliss Lemma to show that all radial extremals of Sobolev inequality must take the form as
\begin{equation*}
U=\big(1+|x|^{\frac{p}{p-1}}\big)^{-\frac{n-p}{p}},
\end{equation*}
up to some dilation and translation. However, they didn't classify  all extremals of Sobolev inequality. Later, Erausquin, Nazaret and Villani \cite{ENV} showed that all extremals must take the form as
\begin{equation*}
U=\big(1+|x|^{\frac{p}{p-1}}\big)^{-\frac{n-p}{p}},
\end{equation*}
up to some dilation and translation by the optimal transportation method. The case
$p = 1$, which was obtained first by Federer-Fleming \cite{FF} and Fleming-Rishel \cite{FR}, is equivalent to the classical isoperimetric inequality on $\mathbb{R}^n$:
\[nv_n^{1/n}|\Omega|^{(n-1)/n}\leq\mathrm{Per}(\Omega),\] where $v_n$ stands for the volume of the unit ball in $\mathbb{R}^n$.

\medskip
 The Sobolev inequality in convex cones has been firstly established in \cite{LPT} and it has been generalized to the anisotropic setting in \cite{CRS}:
 \begin{equation}\label{e0.2}\|u\|_{L^{p^*}(\Sigma)}\leq S_{\Sigma,F}\|F(\nabla u)\|_{L^p(\Sigma)},\end{equation}
 where $F$ is a positive, one-homogenous, convex function (see Section 2 for precise definition ) and $\Sigma$ is a convex open cone in $\mathbb{R}^n$ given by
 \[\Sigma=\{tx:\ x\in w,\ t\in(0,+\infty)\}\]
 for some open domain $w\subset\mathbb{S}^{n-1}$.  The extremals of \eqref{e0.2} are classified in \cite{CFR} and take the form as
$$u=c\sigma^{\frac{n-p}{p}}(1+(\sigma F^o(x-x_0))^{\frac{p}{p-1}})^{-\frac{n-p}{p}},\ c\in \mathbb{R},\ \sigma>0.\ \ $$
 Moreover, if $\Sigma=\mathbb{R}^n$,  then $x_0$ may be any point of $\mathbb{R}^n$; if $\Sigma=\mathbb{R}^k\times\mathcal{C}$ with $k\in\{1,\cdots,n-1\}$ and $\mathcal{C}$ does not contain a line, then $x_0\in\mathbb{R}^k\times\{0\}$; otherwise, $x_0$ is the origin.

\medskip
The classical Sobolev embedding theorem tells us that the Sobolev space $W_0^{1,p}(\Omega)$ is imbedded into $L^q(\Omega)$ for any $1\leq q<\frac{np}{n-p}$. However, when $p=n$, it is known that $W^{1,n}\left( \Omega\right)  \hookrightarrow L^{\infty}\left(
\Omega\right)  $ fails. It it known that the analogue of optimal Sobolev embedding for $W_{0}^{1,n}\left(  \Omega\right)$ (the Sobolev space consisting of functions  vanishing on the boundary $\partial\Omega$) is given by
the famous Trudinger-Moser inequality \cite{T,M2}: \begin{equation}
\mathop {\sup }\limits_{u \in W_0^{1,n}(\Omega),\ \|\nabla u\|_n\leq1} \int_\Omega e^{\alpha_n u^{\frac{n}{n-1}}}dx< +\infty,
\label{TM1}
\end{equation}
where $\alpha_n=n \omega_{n-1}^{\frac{1}{n-1}}$ refers to the sharp constant and $\omega_{n-1}$ denotes the $n-1$ dimensional measure of unit sphere in $\mathbb{R}^n$. The existence of extremals for Moser-Trudinger inequalities was first established by Carleson and Chang \cite{CC} on the unit ball, extended by Flucher \cite{F2} to any bounded domain in dimension $n=2$, and later by Lin \cite{L} to all bounded domains in $\mathbb{R}^n$.

\medskip
To the best of our knowledge, capillary Sobolev and Moser-Trudinger inequalities outside any convex domain have not been studied in the literature. In this paper, we will solve this problem. The key tool in our analysis is the capillary isoperimetric inequality outside any convex domain $E$, established by Fusco et al. in \cite{FJMP}. To state our main result conveniently, we introduce a special capillary gauge $\wF_{\lambda}:\R^n\rightarrow\R$:
 \be\label{e3.2}
 \wF_{\lambda}(\xi)=|\xi|+\nabla h\cdot\xi,
 \ee
where $h$ is a function satisfying the following boundary value problem:
 \be\begin{cases}\label{e3.1}
 -\Delta h=0&\mbox{in\ }  \Omega\subset E^c\\
 \frac{\partial h}{\partial\nu}=\lambda&\mbox{on\ } \partial\Omega\cap\partial E^c
 \end{cases}\ee
 for any $\lambda\in(-1,1)$. In the setting of capillary gauge $\wF_\lambda$, the capillary energy \be\label{e1.15}P_{\lambda}(\Omega;E^c):=\HR^{n-1}(\partial^*\Omega\cap E^c)-\lambda\HR^{n-1}(\partial^*\Omega\cap\partial E) \ee
could be written as anisotropic perimeter functional outside $E$ with respect to $\wF_\lambda$,  where $\partial^*\Omega$ denotes the reduced boundary of $\Omega$. In the rest of this paper, we always assume that $\Omega$  is a bounded domain outside a convex domain $E$ and define
\[W_0^{1,p}(\Omega; E^c):=\{u\in W^{1,p}(\Omega):\ u=0\ \mbox{on}\ \partial\Omega\cap E^c\}\] for $1\leq p<\infty$. We first establish the following capillary P\'{a}lya-Szeg\"{o} rearrangement inequality outside convex domain, which extending the classical Schwartz rearrangement inequality in bounded domain of $\mathbb{R}^n$.

\bt\label{theorem4.1}
(Capillary P\'{o}lya-Szeg\"{o} principle outside convex domain)
 For $1\leq p<\infty$, let $u\in W_0^{1,p}(\Omega;E^c)$ be a non-negative function satisfying the following anisotropic Neumann boundary condition:
\begin{equation}\label{e4.19a}D\wF_{\lambda}(\nabla u)\cdot\nu=0 \mbox{\ on\ } \partial\Omega\cap\partial E^c.\end{equation}
Then,  the following inequality holds:
\be\label{e4.1}
\int_{\Omega}\wF_{\lambda}^p(\nabla u)dx\geq \int_{B_{r}^+(-r\lambda e_n)}F_{\lambda}^p(\nabla u^*)dx,
\ee
where $F_{\lambda}(\xi)=|\xi|-\lambda\left<\xi,e_n\right>$ with $e_n$ being the   n-th coordinate unit vector, and $u^*$  is the capillary Schwartz symmetrization of  $u$ (See Section 2 for definition).
\et

%Let $1\leq p<\infty$, the following inequality known as Sobolev's inequality: assume that $u$ is a function in the Sobolev space $W^{1,p}(\R^n)$, then there exists a constant $C>0$ such that
%\be\label{e1.12}
%\|\nabla u\|_{L^p(\R^n)}\geq C \|u\|_{L^{p^{*}}(\R^n)}.
%\ee
%The best  constant in \eqref{e1.12} is given by
%\be\label{e1.11}
%S=\inf_{u\in W^{1,p}(\R^n)}\frac{\|\nabla u\|_{L^p(\R^n)}}{\|u\|_{L^{p^{*}}(\R^n)}},
%\ee
%where $p^{*}=\frac{np}{n-p}$ is the critical exponent. This constant is also called Sobolev constant and it plays a crucial role in several semilinear elliptic partial differential equations where the right-hand side has a critical growth (see \cite{BN}). It is well-known that the classical Sobolev inequality follows directly from the classical P\'{o}lya-Szeg\"{o} principle. Similarly, the capillary Schwarz symmetrization yields an analogous Sobolev-type inequality outside any convex cylinders.

The symmetrization preserves the $L^{q}$ norm of $u$ for any $q\geq 1$, while the capillary P\'{o}lya-Szeg\"{o} inequality outside convex domain ensures that the $L^p$ norm of the gradient does not increase under the assumption of Neumann boundary condition \eqref{e4.19a}. Consequently, the sharp capillary Sobolev inequality in $W_0^{1,p}(\Omega; E^c)$ with Neuman boundary condition reduces to the sharp capillary Sobolev inequality on the half space which has been established in \cite{LXZ}. However, we still expect that the sharp capillary Sobolev inequality in $W_0^{1,p}(\Omega; E^c)$ still holds without any extra Neuman boundary condition. This will be achieved by using subcritical approximation method. Specifically, we prove that:

\bt\label{theorem4.4}(The sharp Sobolev-type inequality outside convex domain)
 For $1<p<n$, let $u\in W_0^{1,p}(\Omega;E^c)$ be a non-negative function. Then, for any $\lambda\in(-1,1)$, the following Sobolev-type inequality holds:
\be\label{e5.11}\int_{\Omega}\wF_{\lambda}^p(\nabla u) dx\geq C(\lambda,p)^{-p}\left(\int_{\Omega}|u|^{\frac{np}{n-p}}dx\right)^{\frac{n-p}{n}},\ee
where  $C(\lambda,p)$ is the best anisotropic Sobolev constant in the half-space established in \cite{CFR}.
\et

Furthermore, we also establish the following capillary Moser-Trudinger inequality outside convex domain.

\begin{theorem}\label{theorem1.6}(Moser-Trudinger inequality outside convex domain)
Let $\Omega$ be a bounded domain outside a convex domain $E$, and let $u\in W_0^{1,n}(\Omega;E^c)$ be a non-negative function such that
 \[\int_{\Omega}|\wF_{\lambda}(\nabla u)|^ndx\leq 1.\]
Then, it follows that
\[\int_{\Omega}\exp\left(\widetilde{\lambda}_n|u|^{\frac{n}{n-1}}\right)dx\leq C(n),\]
where $\widetilde{\lambda}_n=n(2n\widetilde{\kappa}_n)^{\frac{1}{n-1}}$ and $\widetilde{\kappa}_n=|\{x\in\R^n:F_{\lambda}^o(x)\leq 1\}|$.
\end{theorem}

Capillary Sobolev inequality and Moser-Trudinger inequality outside convex domain we have established in Theorem \ref{theorem4.4} and Theorem \ref{theorem1.6} are in fact can be seen the inequalities defined on Finsler manifold. Recall the notion of a Finsler manifold. Let $M$ be a connected $n$-dimensional $C^{\infty}$ manifold and $TM=\bigcup_{x\in M} T_{x}M$ be its tangent bundle. The pair $(M,F)$ is a Finsler manifold if the continuous function $F:TM \rightarrow[0,\infty)$ satisfies the following conditions
\vskip0.1cm

 \begin{itemize}
\item[(1)] $F\in C^{\infty}(TM\setminus\{0\})$;
\vskip0.1cm

\item[(2)] $F(x,ty)=tF(x,y)$ for all $t\geq0$ and $(x,y)\in TM$;
\vskip0.1cm

\item[(3)] $g_{ij}(x,y):=\left(\frac{1}{2}F^2\right)_{y_i y_j}(x,y)$ is positive definite for all $(x,y)\in TM\setminus\{0\}$.
\end{itemize}
The Finsler metric $F$ is called reversible if $F(x,-y)=F(x,y)$, otherwise it is called nonreversible (or irreversible).  If $g_{ij}(x)=g_{ij}(x,y)$ is independent of $y$,  then $(M,F)$ is called a Riemannian manifold. It follows immediately from the definition that $(\Omega,\widetilde{F}_{\lambda})$ is a reversible Finsler manifold. For more details about Finsler manifold, one can refer to \cite{K2}.
\medskip

Symmetrization is also a useful tool in comparing solutions of elliptic or quasilinear elliptic equation with the Dirichlet boundary or Neuman boundary or Robin boundary condition. The classical Talenti comparison result claims that: if $u$ and $v$ are the solutions of the following two  boundary value problems, respectively:
\[\begin{cases}
-\Delta u=f&\mbox{in\ }\Omega,\\
u=0&\mbox{on\ }\partial\Omega,
\end{cases}\quad\text{and}\quad
\begin{cases}
-\Delta v=f^{\#}&\mbox{in\ }\Omega^{\#},\\
v=0&\mbox{on\ }\partial\Omega^{\#},
\end{cases}
\]
where $f^{\#}$ is the classical Schwartz symmetrization of $f$. Then $u^{\#}(x)\leq v(x)$ for any $x\in\Omega^{\#}$ (centered at the origin with volume $|\Omega|$).

Talenti's comparison results have been widely studied with different operators and various boundary conditions in recent years. Alvino et al. \cite{ANT} established the Talenti comparison principle for Laplace operator with Robin boundary conditions, while Chen-Yang \cite{CY25} extended this result to Finsler Laplacian operator case. The Talenti comparison principle with the Neuman boundary has also been established in \cite{CNT,CY24}. In this paper, we consider the capillary Talenti comparison principle outside convex domain and obtained the following result.

\bt\label{theorem4.2}(Capillary  Talenti comparison principle outside convex domain)
Let $u\in W_0^{1,2}(\Omega; E^c)$ be a weak solution of the following mixed boundary value problem:
\be\label{e4.2}\begin{cases}
-\div\left(\wF_{\lambda}(\nabla u)D\wF_{\lambda}(\nabla u)\right)=f&\mbox{in\ }\Omega\\
u=0&\mbox{on\ }\partial\Omega\cap E^c\\
\wF_{\lambda}(\nabla u)D\wF_{\lambda}(\nabla u)\cdot\nu=0&\mbox{on\ }\partial\Omega\cap\partial E^c,
\end{cases}\ee
where $\wF_{\lambda}$ is the capillary gauge defined in \eqref{e3.2}.
Let $f^*$ denote the capillary Schwartz symmetrization of $f$ and consider the symmetrized problem:
\be\label{e4.4}\begin{cases}
-\div\left(F_{\lambda}(\nabla v)D F_{\lambda}(\nabla v)\right)=f^{*}&\mbox{in\ }B^+_r(-r\lambda e_n)\\
v=0&\mbox{on\ }\partial B_r(-r\lambda e_n)\cap \R_+^n\\
 F_{\lambda}(\nabla v)D F_{\lambda}(\nabla v)\cdot\nu=0&\mbox{on\ } B_r(-r\lambda e_n)\cap\partial \R_+^n.
\end{cases}\ee
 Then, the capillary Schwartz symmetrization $u^{*}$ of $u$ satisfies $u^{*}\leq v$ almost everywhere in $B^+_r(-r\lambda e_n)$. Moreover, the equality holds if and only if $\Omega$ sits on a facet of $\Omega$ and is isometric to $B_r^{+}(-r\lambda e_n)$.
\et

As an direct application of the Capillary  Talenti comparison principle outside convex domain, we can immediately derive the following capillary Bossel-Daners inequality outside convex domain and its rigidity.

%Recall the first eiegenvalue problem of the Laplace operator:
%\[\begin{cases}
%-\Delta u=\lambda_1u\quad&\mbox{in\ }\Omega,\\
% u=0\quad&\mbox{on\ }\partial\Omega,
%\end{cases}\]
%where $\nu$ denotes the outer unit normal to $\partial \Omega$, and $\lambda_1$ denote the  first eiegenvalue of the Laplace operator, i.e.,
%\[\lambda_1=\min_{0\neq u\in W_0^{1,2}(\Omega)}\frac{\int_{\Omega}|\nabla u|^2dx}{\int_{\Omega}u^2dx}.\]

%\noindent It is well-known that the Bossel-Daners inequality (namely the Faber-Krahn inequality) holds as follows
%\[\lambda_1(\Omega)\geq\lambda_1(\Omega^*),\]
%where $\Omega^*$  denotes the ball, centered at the origin, has the same Lebesgue measure with $\Omega$. The above inequality can be proved by the calssical Talenti comparison result.

\bt\label{corollary4.1} (Capillary Bossel-Daners inequality outside convex domain)
Let $\widetilde{\lambda}_1(\Omega;E^c)$ be defined as above, then
\be\label{e5.20}
\widetilde{\lambda}_1(\Omega;E^c)\geq \widetilde{\lambda}_1(B_{r}(-r\lambda e_n);\R^n_+),
\ee
where $\Omega$ is a set of finite perimeter outside some convex domain $E$ and $$\widetilde{\lambda}_1(\Omega;E^c)=\min_{ 0\neq u\in W_0^{1,2}(\Omega;E^c) }\frac{\int_{\Omega}|\wF_{\lambda}(\nabla u)|^2dx}{\int_{\Omega}u^2dx}.$$ Moreover, the equality holds if and only if $\Omega$ sits on a facet of $\Omega$ and is isometric to $B_r^{+}(-r\lambda e_n)$.
\et

% If $F(x,ty)=|t|F(x,y)$ for all $t\in\mathbb{R}$ and $(x,y)\in TM$, then $(M,F)$ is a reversible Finsler manifold. If $g_{ij}(x)=g_{ij}(x,y)$ is independent of $y$,  then $(M,F)$ is called a Riemannian manifold.  The Finsler metric $F$ is called reversible if $F(x,-y)=F(x,y)$, otherwise it is called nonreversible (or irreversible). Clearly, a Riemannian metric is always a reversible Finsler metric.

%Let $F_{\lambda}(x,\xi):T\Omega\setminus\{0\}\rightarrow\mathbb{R}$ be defined by
%\[F_{\lambda}(x,\xi):=|\xi|+\nabla h(x)\cdot\xi \quad\text{ for any } x\in\Omega,\ \xi\in T\Omega\setminus\{0\},\]
%where $h\in C^{\infty}(\overline{\Omega})$ is a harmonic function satisfying the boundary value problem introduced in \eqref{e3.1}. Then, $(\Omega,F_{\lambda}(x,\xi))$ defines a  Finsler manifold. Indeed, the smoothness of $h$ implies that $F_{\lambda}(x,\xi)\in C^{\infty}(T\Omega\setminus\{0\})$. The positive homogeneity of degree one is immediate: for any $t\geq0$,
%\[F_{\lambda}(x,t\xi)=|t\xi|+\nabla h(x)\cdot (t\xi)=t\xi+t(\nabla h(x)\cdot\xi)=tF_{\lambda}(x,\xi).\]
%Moreover, since Hess$\left(\frac{1}{2}F^2(\xi)\right)$ is positive definite on $\mathbb{R}^n\setminus\{0\}$  (as $F$ defines a Minkowski norm for each fixed $x$), it follows that  $\left(\frac{1}{2}F^2\right)_{\xi_i\xi_j}(x,\xi)$ is positive definite in $T\Omega\setminus\{0\}$. Thus, $(\Omega,F_{\lambda}(x,\xi))$ is indeed a Finsler manifold.

This paper is organized as follows: In Section 2, we review some preliminary knowledge about the anisotropic perimeter, the capillary Schwartz symmetrization.
In Section 3, we establish the anisotropic co-area formula outside convex domain and provide the proof of Theorem \ref{theorem4.1}. In Section 4, we establish the sharp capillary Sobolev's inequality and capillary Moser-Trudinger inequality outside  convex domain. In Section 5, we study the capillary Talenti comparison principle outside convex domain and deduce the capillary Bossel-Daners inequality outside convex domain and its rigidity.

%%%%%%  \par首行空两格

  \s{Preliminaries }

  \medskip
In this section, we recall the necessary background on anisotropic perimeters and introduce a reformulation of the capillary isoperimetric inequality outside convex domain established in \cite{FJMP}. This reformulation, based on a suitable gauge function, will play a crucial role in the proof of our main results.

   \subsection{Anisotropic perimeter}

%In fact, the authors further demonstrated that this relative isoperimetric inequality remains valid outside any closed convex set, provided that the set satisfies the $\lambda$-ABP  property. To formalize this property, we introduce the following definition:

%With this definition in place, we can now state the relative isoperimetric inequality outside any convex set, which plays a central role in this work. To formulate this result, we impose regularity conditions on the set $E$ and the domain $\Omega$. Specifically, we assume that:

% \be\label{e1.7}
%E\subset\R^n\ \mbox{is a closed convex set of class }\ C^2
% \ee
 %and
 %\be\label{e1.8}\begin{split}
% \Omega\subset\R^n\backslash E\ \mbox{is a bounded Lipschitz set such that $\Gamma:=\partial\Omega\backslash E$} \\
 %\mbox{is a (n-1)-manifold with boundary of class $C^2$.}\end{split}
% \ee
%\textbf{Theorem B:}(Relative Isopermetric Inequality Outside Convex Sets) Let $E\subset\R^n$ be a closed convex satisfying the regularity condition \eqref{e1.7} and assume that $E$ satisfies the $\lambda$-ABP property for some $\lambda\in(-1,1)$. Then, for every open set $\Omega$ satisfying \eqref{e1.8} and $|\Omega|=|B_r(-r\lambda e_n)\cap\R^n_+|$, the following inequality holds:
%\be\label{e1.9}
%P_{\lambda}(\Omega;E^c)\geq P_{\lambda}(B_r(-r\lambda e_n);\R^n_+).
%\ee
%Furthermore, if $P_{\lambda}(\Omega;E^c)= P_{\lambda}(B(-\lambda e_n);\R^n_+)$, then $\Omega$ coincides with a solid spherical cap isometric to $B(-\lambda e_n)\cap R^n_+$, supported on a facet of $E$.

  \

  \medskip
   We begin by  recalling  the notion of the Wulff ball.  Let $F:\R^n\rightarrow [0,+\infty]$ be a convex function satisfying the homogeneity property:
  \[
  F(tx)=|t|F(x),\quad\forall x\in \R^n,\   \forall t\in\R.
  \]
  Restricting $F$ on $\SR^{n-1}$, the Cahn-Hoffman map $\Phi:\SR^{n-1}\rightarrow\R$ is given by:
  \[\Phi(x):=\nabla F(x).\]
  The image $\Phi(\SR^{n-1})$ is called the Wulff shape.
 The corresponding dual metric of $F$ is defined as:
\[F^o(x)=\sup_{\xi\in K}\left<x,\xi\right>.\]
where $K(x)=\{x\in\R^n:F(x)\leq1\}.$ $F^o(x)$ is also a convex, one-homogeneous function and $F$, $F^o$ are polar to each other in the sense that
\be\label{e2.8}F^o(x)=\sup_{\xi\neq 0}\frac{\left<x,\xi\right>}{F(\xi)},\quad\mbox{and}\quad F(x)=\sup_{\xi\neq 0}\frac{\left<x,\xi\right>}{F^o(\xi)}.\ee
It is clear that $F^o(x)$ is the gauge function of the dual set $K^o$, defined as
\[K^o(x)=\{x\in\R^n:F^o(x)\leq1\},\]
which is also referred to as the unit Wulff ball. We denote the measure of $K^o$ by $\kappa_n$. For convenience, we define the Wulff ball of radius $r$ centered at $x_0$ as
 \[\WR_r(x_0):=r K^o+x_0\]
It is well-known that $F$ and $F^o$ satisfy the following properties:
\be\label{e2.7}F(DF^o(x))=1,\quad DF(x)\cdot x=F(x),\quad\mbox{and\ }F^o(x)DF(DF^o(x))=x.\ee
 For further details, we refer to the literature \cite{AFT,LA,R}.

 \medskip
Now, let $\Sigma$ be an open subset of $\R^n$. For a function $u\in BV(\Sigma)$,  the total variation with respect to $F$  is defined as  (see \cite{AB}):
\[\int_{\Sigma}|Du|_{F}dx=\sup\left\{\int_{\Sigma}u  \div \sigma dx:\sigma\in C_0^1(\Sigma;\R^n), F^o(\sigma)\leq 1\right\}.\]
A Lebesgue measurable set $\Omega\in\R^n$ is said to have locally finite perimeter with respect to $F$, if for every compact set $K\subset\R^n$,
\[\sup\left\{\int_{\Omega} \div \sigma dx:\sigma\in C_0^1(\R^n;\R^n), \mathrm {spt}{~\sigma}\subset K, \sup_{\R^n}F^o(\sigma)\leq 1\right\}<\infty.\]
If this quantity is bounded independently of $K$, then $\Omega$ is said to have finite perimeter in $\R^n$.
The relative anisotropic perimeter of $\Omega$ in $\Sigma$ is defined as
\[P_F(\Omega;\Sigma)=\int_{\Sigma}|D_{\chi_{\Omega}}|_{F}dx=\sup\left\{\int_{\Omega}\div \sigma dx:\sigma\in C_0^1(\Sigma;\R^n), F^o(\sigma)\leq 1\right\}.\]
This anisotropic perimeter, or anisotropic surface energy, can also be expressed as
\be\label{e2.1}P_F(\Omega;\Sigma)=\int_{\partial\Omega^*\cap\Sigma}F(\nu_{\Omega})d\HR^{n-1},\ee
where $\partial\Omega^*$ is the reduced boundary of $\Omega$, and $\nu_{\Omega}$ is the measure-theoretic outer unit normal to $\Omega$ (see \cite{AB}).   Alvino et al. established the following anisotropic isoperimetric inequality by using the Brunn-Minkowski inequality:
\be\label{e2.3}
P_F(\Omega;\R^n)\geq n\kappa_n^{\frac{1}{n}}|\Omega|^{1-\frac{1}{n}}.
\ee
In the special case where $F(\xi)=|\xi|$, the anisotropic perimeter \eqref{e2.1} reduces to the classical perimeter $P(\Omega;\Sigma)$, and the inequality \eqref{e2.3}  simplifies to the classical isoperimetric inequality established by De Giorgi \cite{DG}.

%\par
%In this paper, we focus on an relative isoperimeteric inequality outside convex sets. Let $E\subset\RN$ be a convex set with nonempty interior. Given a set of finite perimeter $\Omega$ in $\RN\backslash E$ denoted by $E^c$, for some $\lambda\in(-1,1)$, we consider the capillary energy:
%\be\label{e2.5}
%P_{\theta}(\Omega; E^c):=P(\Omega; E^c)-\lambda \HR^{N-1}(\partial\Omega\cap\partial E^c).\ee
% Fusco et al. proved the following relative isopermetric inequality outside convex domain $E$:

%\textbf{Theorem B}(relative isopermetric inequality outside convex sets):\\
%Let $\Omega$ be a bounded domain outside a convex set $E$. Assume that $\RN_+=\{x\in\RN:\left<x, e_n\right> >0\}$ be the upper half-space and $e_n=(0,0,\cdot\cdot\cdot, 1)$. For any $\theta\in(0,\pi)$, set the spherical cap $C_{r,\theta}:=B(-r cos\theta E_N, r)\cap\RN_+$  such that $|C_{r,\theta}|=|E|$ , then we have

%\be\label{e2.2}
%P_{\theta}(\Omega; E^c)\geq P_{\theta}(C_{r,\theta}; \RN_+),
%\ee
%where the equality holds if and only if $\Omega=C_{r,\theta}$.

 \subsection{Capillary Schwartz symmetrization outside convex domain}
 \

 \medskip
%Symmetrization  is a classical and influential method  in the study of functional and geometric inequalities. Schwartz symmetrization is a particular kind of rearrangement technique for functions defined on a bounded domain $\Omega\subset\R^n$. It constructs a radially decreasing function on a ball centered at the origin with the same volume as $\Omega$, that preserves the range of the original function  and some essential  geometric and measure-theoretic properties.
  Let $u$ be a measurable function in $\Omega$, for $t\in\R$, the level set $\{u>t\}$ and $\{u=t\}$ is defined as
 \[\{u>t\}:=\{x\in\Omega:u(x)>t\}\quad\text{and}\quad \{u=t\}:=\{x\in\Omega:u(x)=t\},\]
  respectively.   Then, the distribution function of $u$ is defined as
\[\mu (t)=|\{ u>t\}|\quad\mbox{for any }t\geq\mathrm{ess}. \inf(u).\]
 It is easy to derive that $\mu (t)$ is a monotonically decreasing function of $t$ and for $t\geq \mathrm{ess}. \sup(u)$, we have $\mu (t)=0,$ while for $t\leq\mathrm{ess}. \inf(u)$, we have $\mu (t)=|\Omega|$. Thus the range of $\mu $ is the interval $[0,|\Omega|]$.

\medskip
We now introduce the concept of capillary Schwartz symmetrization outside   convex domain.  Let $u:\R^n\setminus E\rightarrow\R$ be a non-negative measurable function that vanishes at infinity. For each $t\geq0$, we  define $r_t$ as the radius of the spherical cap in the upper half-space denote by $B^+_{r_t}(-r_t\lambda e_n)$ such that its volume equals the measure of the super-level set $\mu(t)$. Specifically, we set
 \[\kappa_{\lambda}r_t^n=|B^+_{r_t}(-r_t\lambda e_n)|=\mu(t).\]
 where $\kappa_{\lambda}=|B^+(-\lambda e_n)|$ is a constant depending on $\lambda$.
 The capillary Schwartz symmetrization of $u$, denoted by $u^{*}$ is then defined as
\[u^{*}(x)=\inf\{t\geq 0:\mu(t)<\kappa_{\lambda}|x-(-r_t\lambda e_n)|^n\}=\inf\{t\geq 0:r_t<|x+r_t\lambda e_n|\}.\]
By construction, it is evident that for any $t>0$, the super-level set $\{u^*>t\}$ has the same measure as $\{u>t\}$. This property aligns with the classical idea of Schwartz symmetrization, where the symmetrized function preserves the measure of its super-level sets while rearranging them into a more symmetric form.

  %This equivalence allows us to reformulate the capillary Schwarz symmetrization  $u^{*}$ in terms of the unidimensional decreasing rearrangement of $u$. Specifically, we have
%\[u^{*}=u^{\#}(\kappa_{\lambda}|F^o_{\lambda}(x)|^n),\]
%where $u^{\#}(s)=\inf\{t\geq0:\mu (t)<s\}$. This reformulation highlights the connection between the capillary Schwartz symmetrization and the classical Schwartz symmetrization, while incorporating the geometric and analytical features introduced by the capillary gauge $F_{\lambda}$.

\medskip
In this paper, our analysis relies on the following capillary isoperimetric inequality outside convex domain, which was established in  \cite{FJMP} by the $\lambda$-ABP argument:

\bl\label{lemma2.1} (Capillary isopermetric inequality outside convex domain): Let $\lambda\in(-1,1)$ and let $E$ be a closed convex domain.  For any domain $\Omega\subset\R^n\backslash E$ of finite perimeter with $|\Omega|=|B_r^+(-r\lambda e_n)|$, the following inequality holds:
\be\label{e1.4}
P_{\lambda}(\Omega;E^c)\geq P_{\lambda}(B_r(-r\lambda e_n;\R^n_+).
\ee
 Moreover, the equality holds if and only if $\Omega$  is supported on a facet of  $E$ and is isometric to the spherical cap $B_r^+(-r\lambda e_n)$.
\el

In fact, the capillary surfaces in the upper half-space can be represented as an anisotropic area functional:
 \be\label{e3.3}
 P_{\lambda}(B_r(-r\lambda e_n);\R^n_+)=\int_{\partial B_r(-r\lambda e_n)\cap\R^n_+}F_{\lambda}(\nu)d\HR^{n-1},\ee
 where $F_{\lambda}:\R^n\rightarrow\R$ is a special gauge introduced in \cite{LXZ},  given by
 \be\label{e2.4} F_{\lambda}(\xi)=|\xi|-\lambda\left<\xi, e_n\right>,\ee
and the Euclidean ball $B_{r}(-r\lambda e_n)$ is   precisely  the Wulff ball $\mathcal{W}_r$ centered at the origin with respect to $F_{\lambda}$.

\medskip

We conclude this subsection by showing that the anisotropic gauge $\widetilde{F}_{\lambda}$ defined in \eqref{e3.2} allows  us to rewrite the capillary isopermetric inequality   \eqref{e1.4} as an anisotropic isopermetric inequality. This representation will play a central role in our subsequent analysis. It is straightforward to verify that $\wF_{\lambda}$ is a one-homogeneous function and is smooth on $\R^n\backslash\{0\}$.
 Since
 \[\nabla \wF_{\lambda}(\xi)=\frac{\xi}{|\xi|}+\nabla h,\]
 then the Wulff shape associated with  $\wF_{\lambda}$ is given by:
 \[\nabla \wF_{\lambda}(\SR^{n-1})=\SR^{n-1}+\nabla h=\{|x-\nabla h|=1\}.\]
 \bo\label{proposition3.2}
 The dual gauge $\wF_{\lambda}^o:\R^n\rightarrow\R$ is given by:
 \[\wF_{\lambda}^o(x)=\frac{|x|^2}{\sqrt{\left<x\cdot\nabla h\right>^2+|x|^2(1-|\nabla h|^2)}+\left<x,\nabla h\right>}.\]
 \eo
 \bp
 Consider the convex set $\widetilde{K}$ determined by $\nabla \wF_{\lambda}(\SR^{n-1})=\{|x-\nabla h|=1\}$. We want to find the radial function of $\widetilde{K}$. Let $y\in\nabla \wF_{\lambda}(\SR^{n-1})$ be given by $y=\rho(x)x$, where $\ x\in\SR^{n-1}$. Then,
 \[|\rho(x)x-\nabla h|=1,\]
which implies that
\[\rho(x)=\sqrt{\left<x\cdot\nabla h\right>^2+1-|\nabla h|^2}+\left<x,\nabla h\right>.\]
Thus, $\rho:\SR^{n-1}\rightarrow\R$ is the radial function of $\widetilde{K}$. By a classical result, the support function of $\widetilde{K}^o$  equals to the radial function of $\widetilde{K}$ (see e.g.,\cite{RS}).
Therefore,
\[\wF^o_{\lambda}(x)=\frac{1}{\rho(x)}=\frac{1}{\sqrt{\left<x\cdot\nabla h\right>^2+1-|\nabla h|^2}+\left<x,\nabla h\right>}.\]
Extending $\wF^o_{\lambda}$ to $\R^n$ via one-homogeneity, i.e., $\wF^o_{\lambda}(x)=|x|\wF^o_{\lambda}(\frac{x}{|x|})$. Then we complete the proof.
 \ep

\br\label{remark2.2}
If we consider the problem \eqref{e3.1} in the half-space:
 \[\begin{cases}
 -\Delta h=0&\mbox{in\ } \R^n_+\\
 \frac{\partial h}{\partial \nu}=\lambda&\mbox{on\ } \partial\R^n_+,
 \end{cases}\]
 then it is easy to check that $h=-\lambda x_n$ and $\nabla h=-\lambda e_n$. Therefore, in the case of $E^c=\R^n_+$, we find that
 \[\wF_{\lambda}(\xi)=|\xi|-\lambda\left<\xi, e_n\right>=F_{\lambda}(\xi).\]
\er

 \bo\label{proposition3.1}
For any $\lambda\in(-1,1)$, the capillary surface outside convex domain can be rewritten as an anisotropic perimeter as follows
 \[P(\Omega;E^c)-\lambda\HR^{n-1}(\partial^*\Omega\cap\partial E)=\int_{\partial\Omega\cap E^c}\wF_{\lambda}(\nu)d\HR^{n-1}.\]
 \eo
 \bp
 Using the divergence theorem in equation \eqref{e3.1}, we obtain
 \[\begin{split}
 0=\int_{\Omega}\Delta hdx=&\int_{\partial\Omega\cap E^c}\frac{\partial h}{\partial\nu}d\HR^{n-1}+\int_{\partial\Omega\cap \partial E^c}\frac{\partial h}{\partial\nu}d\HR^{n-1}\\
 =&\int_{\partial\Omega\cap E^c}\frac{\partial h}{\partial\nu}d\HR^{n-1}+\lambda\HR^{n-1}(\partial\Omega\cap\partial E^c).
 \end{split}\]
This implies
 \[\int_{\partial\Omega\cap E^c}\frac{\partial h}{\partial\nu}d\HR^{n-1}=-\lambda\HR^{n-1}(\partial\Omega\cap\partial E^c).\]
By the definition of $\wF_{\lambda}$, we have
 \be\label{e3.4}\begin{split}
 P_{\lambda}(\Omega; E^c)=&P(\Omega; E^c)+\int_{\partial\Omega\cap E^c}\frac{\partial h}{\partial\nu}d\HR^{n-1} \\
  =&\int_{\partial\Omega\cap E^c}\left(|\nu|+\frac{\partial h}{\partial\nu}\right)d\HR^{n-1}\\
 =&\int_{\partial\Omega\cap E^c}\wF_{\lambda}(\nu)d\HR^{n-1},\end{split}\ee
which completes the proof.
 \ep

  In this way, the isoperimeteric inequality \eqref{e1.4} can be rewritten as the  capillary Choe-Ghomi-Ritor\'{e} relative isoperimetric inequality outside convex domain for the special gauge $\wF_{\lambda}$ and $F_{\lambda}$. Specifically, let $\lambda\in(-1,1)$ and $E\subset\R^n$ be a closed convex set. Then, for every open set $\Omega\subset\R^n\backslash E$ such that $|\Omega|=|B_r^+(-r\lambda e_n)|$, the isoperimeteric inequality can be expressed as
\be\label{e3.6}P_{\wF_{\lambda}}(\Omega; E^c)\geq P_{F_{\lambda}}(B_r(-r\lambda e_n);\R^n_+),\ee
where $\wF_{\lambda}$ and $F_{\lambda}$ are defined by \eqref{e3.2} and \eqref{e2.4}, respectively.
Moreover, the equality holds if and only if $\Omega$ coincides with a solid spherical cap that is isometric to $B^+_r(-r\lambda e_n)$, supported on a facet of $E$.

\s{Proof of the Theorem \ref{theorem4.1}}

In this section, we are concerned with the anisotropic co-area formula outside convex domain and give the proof of capillary P\'{o}lya-Szeg\"{o} principle outside convex domain.
The co-area formula is a fundamental tool in geometric measure theory and analysis, providing a relationship between integrals over level sets of a function and its gradient.  In  general, if $u:\R^n\rightarrow\R$ is a Lipschitz continuous function with ess$\ \inf |\nabla u|\neq 0$, and $g:\R^n\rightarrow\R$ is an integrable function,  then
\[\int_{\{u>t\}}gdx=\int_t^{\infty}\left(\int_{\{u=\tau\}}\frac{g}{|\nabla u|}d\HR^{n-1}\right)d\tau.\]
In particular, differentiating with respect to $t$ yields
\[-\frac{d}{dt}\left(\int_{\{u>t\}}g dx\right)=\int_{\{u=t\}}\frac{g}{|\nabla u|}d\HR^{n-1}.\]
This formula will be instrumental in our analysis, especially when $g(x)=F^p(\nabla u)$, where $F$ is a gauge function and $p\geq 1$.

\bl\label{theorem4.5}
Let $u\in W^{1,p}(\Omega; E^c)$ be a non-negative function satisfying the following boundary value problem:
\be\label{e4.13}\begin{cases}
-\div\left( F^{p-1} (\nabla u)D F (\nabla u)\right)=f&\mbox{in\ }\Omega\\
u=0&\mbox{on\ }\partial\Omega\cap E^c\\
 D F (\nabla u)\cdot\nu=0&\mbox{on\ }\partial\Omega\cap\partial E^c,
\end{cases}\ee
where $F$ is the gauge defined by \eqref{e2.8} and $\nu$ is the unit normal vector to $\partial\Omega$.
Then for $1\leq p<\infty$, the following identity holds:
\be\label{e4.14}-\frac{d}{dt}\int_{\{u>t\}} F ^p(\nabla u)dx=\int_{\{u=t\}}\frac{ F ^p(\nabla u)}{|\nabla u|}d\HR^{n-1},\ee
where $\{u>t\}:=\{x\in\Omega:u(x)>t\}$ and $\{u=t\}:=\{x\in\Omega:u(x)=t\}$.
\el

\bp
Define $f=-\div\left( F^{p-1}  (\nabla u)D F (\nabla u)\right)$,  for any test function $\psi\in W_0^{1,p}(\Omega; E^c)$, we have
\be\label{e4.15}\int_{\Omega} F^{p-1} (\nabla u)D F (\nabla u)\cdot\nabla \psi dx-\int_{\partial\Omega\cap E^c} F^{p-1} (\nabla u)D F (\nabla u)\cdot\nu\psi dx=\int_{\Omega}f\psi dx,\ee
where we used the boundary condition that $D F (\nabla u)\cdot\nu=0$ on $\partial\Omega\cap\partial E^c$. Let $t>0$ and choose $\psi=(u-t)^+$, which belongs to $\in W_0^{1,p}(\Omega; E^c)$, and is supported on $\{u>t\}$. Substituting $\psi$ into the above identity \eqref{e4.15}, we obtain
\[\int_{\{u>t\}} F^{p} (\nabla u)dx=\int_{\{u>t\}}f(u-t)dx.\]
Differentiating with respect to $t$ yields
\be\label{e4.16}
-\frac{d}{dt}\int_{\{u>t\}} F^{p} (\nabla u)dx=\int_{\{u>t\}}fdx.
\ee
Observe that the boundary decomposes as
 \[\begin{aligned}\partial\{u>t\}=&(\partial\{u>t\}\cap\Omega)\cup(\{u>t\}\cap\partial\Omega\cap E^c)\cup(\{u>t\}\cap\partial E)\\=&\{u=t\}\cup(\{u>t\}\cap\partial E),\end{aligned}\]
where we used the fact that $\{u > t\} \cap \partial\Omega \cap E^c = \emptyset$ since $u = 0$ on $\partial\Omega \cap E^c$.
Hence, by the definition of $f$ and the divergence theorem, we  have
\[\begin{split}
-\frac{d}{dt}\int_{\{u>t\}}F^{p}(\nabla u)dx
=&\int_{\{u>t\} }-\div\left(F^{p-1}(\nabla u)DF(\nabla u)\right)dx\\
=&-\int_{\partial\{u>t\}}F^{p-1}(\nabla u)DF(\nabla u)\cdot\nu d\HR^{n-1}\\
=&-\int_{\{u=t\} }F^{p-1}(\nabla u)DF(\nabla u)\cdot\nu d\HR^{n-1}
-\int_{\{u>t\}\cap\partial E}F^{p-1}(\nabla u)DF(\nabla u)\cdot\nu d\HR^{n-1}\\
=&\int_{\{u=t\} }\frac{F^p(\nabla u)}{|\nabla u|} d\HR^{n-1}.
\end{split}\]
Here, we used the fact that $\nu=-\frac{\nabla u}{|\nabla u|}$  on the level set $\{u=t\}$, since $u$ is constant on this surface and $\{u>t\}$ in its interior. This completes the proof.
\ep

\bl\label{theorem4}
Let   $u\in W_0^{1,p}(\Omega;E^c)$ be a non-negative function satisfying
\[D F (\nabla u)\cdot\nu=0\quad\mbox{on\ }\partial\Omega\cap\partial E^c.\]
  Then. for almost every $t$ in the range of $u$, we have
\be\label{e4.17}
-\mu'(t)=\int_{\{u=t\}}\frac{1}{|\nabla u|}d\HR^{n-1}=\int_{\{u^*=t\}}\frac{1}{|\nabla u^*|}d\HR^{n-1},
\ee
where $u^*$ is the capillary Schwartz symmetrization of $u$ and  $\mu$ is the distribution function of $u$.
\el
\bp
Let $\e>0$, define
\[f=-\div\left(\frac{DF(\nabla u)}{F(\nabla u)+\e}\right).\]
Multiplying by $(u-t)^+$ and integrating by parts, we obtain
\[\int_{\{u>t\}} \frac{F(\nabla u)}{F(\nabla u)+\e}dx=\int_{\{u>t\}}f(u-t)dx.\]
Differentiating with respect to $t$ yields
\be\label{e4.18}-\frac{d}{dt}\int_{\{u>t\}} \frac{F(\nabla u)}{F(\nabla u)+\e}dx=\int_{\{u>t\}}f dx.\ee
For sufficiently small $h>0$, integrating from $t-h$ to $t$ on both sides of \eqref{e4.18}, we get
\[\begin{split}
\int_{\{t-h<u\leq t\}}\frac{F(\nabla u)}{F(\nabla u)+\e}dx=&\int_{t-h}^t\left(\int_{\{u>\tau\}}fdx\right)d\tau\\
=&\int_{t-h}^t\left(\int_{\{u>\tau\}}-\div\left(\frac{DF(\nabla u)}{F(\nabla u)+\e}\right) dx\right)d\tau\\
=&-\int_{t-h}^t\left(\int_{ \{u=\tau\}}\frac{DF(\nabla u)}{F(\nabla u)+\e}\cdot\nu dx\right)d\tau-\int_{t-h}^t\left(\int_{ \{u>\tau\}\cap\partial E}\frac{DF(\nabla u)}{F(\nabla u)+\e}\cdot\nu dx\right)d\tau\\
=&\int_{t-h}^t\left(\int_{ \{u=\tau\}}\frac{F(\nabla u)}{F(\nabla u)+\e}\frac{1}{|\nabla u|}dx\right)d\tau.
\end{split}\]
Applying the dominated convergence theorem and taking the limit as $\e\rightarrow 0$, we obtain
\[\mu(t-h)-\mu(t)=\int_{t-h}^t\left(\int_{ \{u=\tau\}}\frac{1}{|\nabla u|}d\HR^{n-1}\right)d\tau.\]
Dividing by $h$ and taking the limit as $h\rightarrow 0$, the first equality in \eqref{e4.17} holds.
\par Let $r(t)$ be the radius of the  spherical cap $\{u^*>t\}$. Then $\mu(t)=\kappa_{\lambda}\left(r(t)\right)^n$,
 and so $\mu'(t)=n\kappa_{\lambda}(r(t))^{n-1}r'(t)$. Since $\mu(t)$ and $r(t)$ are monotonically decreasing functions, they are differentiable for almost every  $t$. Noting that $u^{*}(r(t))=t$, then we have $\HR^{n-1}(\{u^{*}=t\})=n\kappa_{\lambda}(r(t))^{n-1}$ and $r'(t)=\frac{1}{(u^{*})'(r(t))}$. Moreover, by implicit differentiation, we have $(u^{*})'=-|\nabla u^{*}|$. It follows that
\[\mu'(t)=\int_{\{u^{*}=t\}}\frac{1}{(u^{*})'(r(t))}d\HR^{n-1}=-\int_{\{u^{*}=t\}}\frac{1}{|\nabla u^{*}|}d\HR^{n-1}.\]
Then, we complete the proof.
\ep

 \par Applying the above results, we can now provide the proof of the capillary P\'{o}lya-Szeg\"{o} principle outside  convex domain.

\noindent\begin{proof}[\bf{Proof of Theorem \ref{theorem4.1}}]:
%Step 1. The case $p=1$.\\
%From co-area formula, isoperimeteric inequality \eqref{e2.2}, \eqref{e3.3} and \eqref{e3.4}, we have
%\[\begin{split}
%\int_{\Omega}\wF_{\theta}(\nabla u)dx=&\int_0^{\infty}\left(\int_{\partial\{u>t\}}\frac{\wF_{\theta}(\nabla u)}{|\nabla u|}d\HR^{N-1}\right)dt\\
%=&\int_0^{\infty}\left(\int_{\partial\{u>t\}}\wF_{\theta}(\nu_{\{u>t\}})d\HR^{N-1}\right)dt\\
%=&\int_0^{\infty} P_{\wF_{\theta}}(\{u>t\},\Omega)dt\\
%\leq &\int_0^{\infty} P_{F_{\theta}}(\{u^*>t\},C_{r,\theta})dt
%=\int_{C_{r,\theta}}F_{\theta}(\nabla u^*)dx.
%\end{split}\]
Since $\wF_{\lambda}$ is a special gauge,  by Lemma \ref{theorem4.5}, we know that for any $t>0$,
\[-\frac{d}{dt}\int_{\{u>t\}}\wF_{\lambda}^p(\nabla u)dx=\int_{\{u=t\}}\frac{\wF_{\lambda}^p(\nabla u)}{|\nabla u|}d\HR^{n-1}.\]
Similarly, for the symmetrized function $u^{*}$, we  have
\[-\frac{d}{dt}\int_{\{u^{*}>t\}} F_{\lambda}^p(\nabla u^{*})dx=\int_{\{u^{*}=t\}}\frac{ F_{\lambda}^p(\nabla u^{*})}{|\nabla u^{*}|}d\HR^{n-1}.\]
On the other hand,  we apply   H\"{o}lder inequality to the integral over the level set $\{u=t\}\cap E^c$. Specifically, we obtain
\[\int_{\{u=t\}}\frac{\wF_{\lambda}(\nabla u)}{|\nabla u|}d\HR^{N-1}\leq\left(\int_{\{u=t\}\cap E^c}\frac{\wF_{\lambda}^p(\nabla u)}{|\nabla u|}\right)^{\frac{1}{p}}\left(\int_{\{u=t\}\cap E^c}\frac{1}{|\nabla u|}\right)^{1-\frac{1}{p}}.\]
This inequality implies
\[\begin{split}
\int_{\{u=t\}}\frac{\wF_{\lambda}^p(\nabla u)}{|\nabla u|}d\HR^{n-1}\geq& \left(\int_{\{u=t\}}\wF_{\lambda}\left(\frac{\nabla u}{|\nabla u|}\right)d\HR^{n-1}\right)^p\left(-\mu'(t)\right)^{1-p}\\
=&\left(P_{\wF_{\lambda}}(\{u>t\};E^c)\right)^p\left(-\mu'(t)\right)^{1-p}\\
\geq&\left(P_{ F_{\lambda}}(\{u^*>t\};\R^n_+)\right)^p\left(-\mu'(t)\right)^{1-p}
=\int_{\{u^*=t\}}\frac{F_{\lambda}^p(\nabla u^*)}{|\nabla u^*|}d\HR^{n-1},
\end{split}\]
where in the last step, we used the conclusion of Lemma \ref{theorem4}, which relates the distribution function $\mu(t)$ to the level sets of $u$ and $u^{*}$. Combining these results, we deduce that
\[-\frac{d}{dt}\int_{\{u>t\}}\wF_{\lambda}^p(\nabla u)dx\geq-\frac{d}{dt}\int_{\{u^*>t\}}F_{\lambda}^p(\nabla u^*)dx.\]
Integrating both sides from $0$ to $+\infty$ with respect to $t$, we   complete  the proof.
\end{proof}

\vspace{-0.1cm}

\s{Proof of Theorem \ref{theorem4.4} and   Theorem \ref{theorem1.6}}

In this section, we establish sharp capillary Sobolev and Moser-Trudinger inequalities outside convex domain. Our analysis is based on the capillary P\'{o}lya-Szeg\"{o} principle established in Theorem \ref{theorem4.1}, subcritical approximation method and subcritical capillary Moser-Trudinger inequalities with the Neuman boundary condition. We first provide a proof for capillary Sobolev inequality with Neumann boundary condition.

%\bo\label{theorem4.3}(\cite{CFR})
%Given $\lambda\in(-1,1)$ and $1<p<n$, let $u$ be a non-negative function in the Sobolev space $\widetilde{W}^{1,p}(\R^n_+)$, defined as
%\[ \widetilde{W}^{1,p}(\R^n_+):=\{u\in L^{\frac{np}{n-p}}(\R^n_+):\nabla u\in L^p(\R^n_+)\}.\]
%Then, the following inequality holds:
%\[\|u\|_{L^{\frac{np}{n-p}}(\R^n_+)}\leq C(\lambda,p)\left(\int_{\R^n_+}(|\nabla u|-\lambda\left<\nabla u, e_n\right>)^p\right)^{\frac{1}{p}},\]
%where, $C(\lambda,p)$ is a constant given by
%\begin{equation}\label{e4.19}C(\lambda,p)=\left(\int_{\R^n_+}(|\nabla U_{\lambda,p}|-\lambda\left<\nabla U_{\lambda,p}, e_n\right>)^p\right)^{-\frac{1}{p}}.\end{equation}
%Here, $U_{\lambda,p}$ is the extremal function defined as
%\[U_{\lambda,p}=-\left(\frac{1}{\sigma_{p,\lambda}+F^o_{\lambda}(x)^{\frac{p}{p-1}}}\right)^{\frac{n-p}{p}},\]
%and $\sigma_{p,\lambda}>0$ is  a constant  determined by the normalization condition $\|U_{\lambda,p}\|_{L^{\frac{np}{n-p}}(\R^n_+)}=1$. Equality holds if and only if
%\[u(x)=CU_{\lambda,p}(\lambda(x-x_0))\]
%for some constant $C\leq0,\lambda\neq 0$ and some point $x_0\in\partial\R^n_+.$
%\eo

%As a direct corollary of Theorem \ref{theorem4.1} and Proposition \ref{theorem4.3}, we have the following result under a Neumann boundary condition.
\bl\label{corollary4.4}
 For $1<p<n$, let $u\in W_0^{1,p}(\Omega;E^c)$ be a non-negative function satisfying
\begin{equation}\label{e4.20} D \wF_{\lambda} (\nabla u)\cdot\nu=0\quad\mbox{on\ }\partial\Omega\cap\partial E^c.\end{equation}
 Then, for any $\lambda\in(-1,1)$, the following Sobolev inequality holds:
\[\int_{\Omega}|\wF_{\lambda}(\nabla u)|^pdx\geq C(\lambda,p)\left(\int_{\Omega}|u|^{\frac{np}{n-p}}dx\right)^{\frac{n-p}{n}}.\]
\el
\bp
From Theorem \ref{theorem4.1}, we have the inequality
\[\frac{\int_{\Omega}|\wF_{\lambda}(\nabla u)|^pdx}{\left(\int_{\Omega}|u|^{\frac{np}{n-p}}dx\right)^{\frac{n-p}{n}}}\geq\frac{\int_{B_r^+(-r\lambda e_n)}|F_{\lambda}(\nabla u^*)|^pdx}{\left(\int_{B_r^+(-r\lambda e_n)}|u^*|^{\frac{np}{ n-p}}dx\right)^{\frac{n-p}{n}}}.\]
Applying the anisotropic sharp  Sobolev inequality in the upper half-space, we obtain
\[\int_{B_r^+(-r\lambda e_n)}|F_{\lambda}(\nabla u^*)|^pdx\geq C^{-p}(\lambda,p)\left(\int_{B_r^+(-r\lambda e_n)}|u^*|^{\frac{np}{ n-p}}dx\right)^{\frac{n-p}{n}},\]
where $C(\lambda,p)$ is the best constant established in \cite{CFR}.    Equivalently, this can be rewritten as
    \[\int_{\Omega}|\wF_{\lambda}(\nabla u)|^pdx \geq C^{-p}(\lambda,p)\left(\int_{\Omega}|u|^{\frac{np}{n-p}}dx\right)^{\frac{n-p}{n}}.\]
This completes the proof.
\ep
In fact, the   anisotropic Neumann boundary condition \eqref{e4.20} can be removed by a subcritical approximation argument.

\begin{proof}[{\bf Proof of Theorem \ref{theorem4.4}}]:
  The argument is divided into two parts.

\medskip
Step 1. We first establish the anisotropic Poincar\'{e} inequality for functions  $u\in W_0^{1,p}(\Omega; E^c)$:
\be\label{e5.7}
\int_{\Omega}|u|^pdx\leq C\int_{\Omega}|\wF_{\lambda}(\nabla u)|^pdx.
\ee
Assume by contradiction that for any $k\geq 1$, there exists a sequence $u_k\in W_0^{1,p}(\Omega;E^c)$   such that
\[\int_{\Omega}|u_k|^p dx > k\int_{\Omega}|\wF_{\lambda}(\nabla u_k)|^p dx.\]
Define the normalized function $w_k(x)=\frac{u_k(x)}{\|u_k\|_{L^p(\Omega)}}$, which satisfies:
\be\label{e5.1}
w_k(x)=0\quad\mbox{on\ }\partial\Omega\cap E^c,
\ee
\be\label{e5.2}
\|w_k\|_{L^p(\Omega)}=1
\ee
and
\be\label{e5.3}
\int_{\Omega}|\wF_{\lambda}(\nabla w_k)|^pdx=\frac{1}{\|u_k\|_{L^p(\Omega)}^p}\int_{\Omega}|\wF_{\lambda}(\nabla u_k)|^pdx<\frac{1}{k}.
\ee
 Since $\alpha|\xi|\leq\wF_{\lambda}(\xi)\leq\beta|\xi|$ for any $\xi\in\R^n$ and some positive constants $\alpha\leq\beta$, then the inequalities \eqref{e5.2} and \eqref{e5.3}  indicate that $\|w_k\|_{W^{1,p}(\Omega)}$ is bounded. According to the compact embedding $W^{1,p}(\Omega)\hookrightarrow L^p(\Omega)$, there exists  $w\in W^{1,p}(\Omega)$ and a subsequence (still denoted by $w_k$) such that
\be\label{e5.4} w_k\rightarrow w (k\rightarrow\infty)\quad \mbox{in\ } L^{p}(\Omega),\ee
and
\be\label{e5.5} Dw_k\rightharpoonup Dw (k\rightarrow\infty)\quad \mbox{in\ } L^{p}(\Omega).\ee
From \eqref{e5.3} and \eqref{e5.5} we deduce $Dw=0$ a.e. $x\in\Omega$, implying $w$ is a constant. The boundary condition \eqref{e5.1} and and strong convergence \eqref{e5.4} yield
\be\label{e5.6} w=0 \quad\mbox{a.e.\ } x\in\overline{\Omega},\ee
where we used  the continuity of $w$.
Using again \eqref{e5.4} and by \eqref{e5.2}, we can deduce that
\[\|w\|_{L^p(\Omega)}=1,\]
which is a contradiction with \eqref{e5.6}.
Therefore, \eqref{e5.7} holds, and for any $u\in W_0^{1,p}(\Omega;E^c)$, $\int_{\Omega}|\wF_{\lambda}(\nabla u)|^p dx$ is equivalent to $\int_{\Omega}|\wF_{\lambda}(\nabla u)|^p dx+\int_{\Omega}|u|^p dx$. Then,
\be\label{e5.8}A_k=\inf_{u\in W_0^{1,p}(\Omega; E^c)}\frac{\int_{\Omega}|\wF_{\lambda}(\nabla u)|^p dx}{\|u\|_{p_k}^p}\quad  \mbox{\ for any\ } p_k<p^{*},\ee
is well-defined, where $p^{*}=\frac{np}{n-p}$ is the critical index. By Sobolev embedding Theorem, $A_k$ can be achieved by some $u_{p_k}$ with $p_k<p^{*}$.

\medskip
Step 2. We claim that
 \be\label{e5.13}
A=\inf_{u\in W_0^{1,p}(\Omega; E^c)}\frac{\int_{\Omega}|\wF_{\lambda}(\nabla u)|^{p}dx}{\left(\int_{\Omega}|u|^{p^{*}}dx\right)^{\frac{p}{p^{*}}}}=C^{-p}(\lambda,p)
\ee
is the best Sobolev constant outside convex domain. Through standard variational arguments, the function $u_{p_k}\in W_0^{1,p}(\Omega;E^c)$ satisfying
\be\label{e5.10}
\begin{cases}
-\div\left(\wF^{p-1}_{\lambda}(\nabla u)D\wF_{\lambda}(\nabla u)\right)=\lambda_k|u|^{p_{k}-2}u&\mbox{in\ }\Omega,\\
u=0&\mbox{on\ }\partial\Omega\cap E^c\\
\wF^{p-1}_{\lambda}(\nabla u)D\wF_{\lambda}(\nabla u)=0&\mbox{on\ }\partial\Omega\cap\partial E^c,
\end{cases}\ee
where $\lambda_k$ denotes the associated Lagrange multiplier. Applying H\"{o}lder inequality, we derive
%Then, by the capilary  P\'{o}lya-Szeg\"{o} principle \eqref{e4.1}, we have
%\be\label{e5.12}
%A_k=\frac{\int_{\Omega}|\wF_{\lambda}(\nabla u_p)|^pdx}{\left(\int_{\Omega}|u_p|^{p_k}dx\right)^{\frac{p}{p_k}}}
%\geq\frac{\int_{B_r(-r\lambda e_n)}|F_{\lambda}(\nabla u_p^{*})|^pdx}{\left(\int_{B_r(-r\lambda e_n)}|u^{*}_p|^{p_k}dx\right)^{\frac{p}{p_k}}}\ee
\[\left(\int_{\Omega}|u_{p_k}|^{p_k}dx\right)^{\frac{p}{p_k}}\leq\left(\int_{\Omega}|u_{p_k}|^{p^{*}}dx\right)^{\frac{p}{p^{*}}}|\Omega|^{\frac{p}{p_k}-\frac{p}{p^{*}}}.\]
 Combining this with the above H\"{o}lder inequality, we get
\[\begin{split}
\lim_{k\rightarrow\infty}A_k=&\lim_{k\rightarrow\infty}\frac{\int_{\Omega}|\wF_{\lambda}(\nabla u_{p_k})|^{p}dx}{\left(\int_{\Omega}|u_{p_k}|^{p_k}dx\right)^{\frac{p}{p_k}}}\\
\geq&\lim_{k\rightarrow\infty}\frac{\int_{\Omega}|\wF_{\lambda}(\nabla u_{p_k})|^{p}dx}{\left(\int_{\Omega}|u_{p_k}|^{p^{*}}dx\right)^{\frac{p}{p^{*}}}}|\Omega|^{\frac{p}{p^{*}}-\frac{p}{p_k}}
=\lim_{k\rightarrow\infty}\frac{\int_{\Omega}|\wF_{\lambda}(\nabla u_{p_k})|^{p}dx}{\left(\int_{\Omega}|u_{p_k}|^{p^{*}}dx\right)^{\frac{p}{p^{*}}}}.
\end{split}\]

On the other hand,  for any $u\in W_0^{1,p}(\Omega;E^c)$, we have
\[\lim_{k\rightarrow\infty}A_k\leq \lim_{k\rightarrow\infty}\frac{\int_{\Omega}|\wF_{\lambda}(\nabla u)|^{p}dx}{\left(\int_{\Omega}|u|^{p_k}dx\right)^{\frac{p}{p_k}}}=\frac{\int_{\Omega}|\wF_{\lambda}(\nabla u)|^{p}dx}{\left(\int_{\Omega}|u|^{p^{*}}dx\right)^{\frac{p}{p^{*}}}}.\]
The arbitrariness of $u$ implies that
\[\lim_{k\rightarrow\infty}A_k\leq \inf_{u\in W_0^{1,p}(\Omega; E^c)}\frac{\int_{\Omega}|\wF_{\lambda}(\nabla u)|^{p}dx}{\left(\int_{\Omega}|u|^{p^{*}}dx\right)^{\frac{p}{p^{*}}}}.\]
Therefore,
$\lim_{k\rightarrow\infty}A_k=A$.
Since $u_{p_k}$ is a solution to the equation \eqref{e5.10}, applying P\'{o}lya-Szeg\"{o} principle stated in \eqref{e4.1}, we can conclude that
\begin{equation}\label{e4.20}A=\lim_{k\rightarrow\infty}A_k\geq \frac{\int_{\Omega}|\wF_{\lambda}(\nabla u_{p_k})|^{p}dx}{\left(\int_{\Omega}|u_{p_k}|^{p^{*}}dx\right)^{\frac{p}{p^{*}}}}\geq \frac{\int_{B^+_r(-r\lambda e_n)}|F_{\lambda}(\nabla u_{p_k}^{*})|^pdx}{\left(\int_{B^+_r(-r\lambda e_n)}|u^{*}_{p_k}|^{p^{*}}dx\right)^{\frac{p}{p^{*}}}}\geq C^{-p}(\lambda,p),\end{equation}
where $C(\lambda,p)$ is the best Sobolev constant in the half-space established in \cite{CFR}.
\vskip0.2cm

Now, in order to prove the sharpness of capillary Sobolev inequality outside convex domain, it only suffices to show that $A\leq C^{-p}(\lambda,p)$. For any fixed $x_0\in\Omega$ and and sufficiently small  $\delta>0$, we have $B_{2\delta}(x_0)\subset\Omega$. Without loss of generality, we assume that $x_0=0$. Let $u_{\varepsilon}(x)=\varepsilon^{-\frac{n-p}{p}}U_{1,0}(\frac{x}{\varepsilon})$, where $U_{1,0}$ is the extremal function in the half-space. Let $0\leq\eta\leq1$ be a cut-off function such that $\eta=1$ on $B_{\delta}$ and $\eta=0$ on $\R^n\setminus B_{2\delta}$.
 Then, we have
\[\begin{aligned}
A=&\inf_{u\in W_0^{1,p}(\Omega; E^c)}\frac{\int_{\Omega}|\wF_{\lambda}(\nabla u)|^{p}dx}{\left(\int_{\Omega}|u|^{p^{*}}dx\right)^{\frac{p}{p^{*}}}}\\
\leq& \frac{\int_{B_{2\delta}\cap\R^n_+} |\wF_{\lambda}(\nabla (\eta u_{\varepsilon}))|^{p}dx}{\left(\int_{B_{2\delta}\cap\R^n_+} |\eta u_{\varepsilon}|^{p^{*}}dx\right)^{\frac{p}{p^{*}}}}\\
\leq& \frac{\int_{B_{2\delta}\cap\R^n_+}|\wF_{\lambda}(\eta\nabla  u_{\varepsilon})|^{p}dx}{\left(\int_{B_{2\delta}\cap\R^n_+} |\eta u_{\varepsilon}|^{p^{*}}dx\right)^{\frac{p}{p^{*}}}}+\frac{\int_{B_{2\delta}\cap\R^n_+}|\wF_{\lambda}(u_{\varepsilon}\nabla \eta)|^{p}dx}{\left(\int_{B_{2\delta}\cap\R^n_+}|\eta u_{\varepsilon}|^{p^{*}}dx\right)^{\frac{p}{p^{*}}}}\\
\leq& \frac{\int_{B_{2\delta}\cap\R^n_+}|\wF_{\lambda}(\nabla  u_{\varepsilon})|^{p}dx}{\left(\int_{B_{\delta}\cap\R^n_+}| u_{\varepsilon}|^{p^{*}}dx\right)^{\frac{p}{p^{*}}}}+C\delta^{-p}\frac{\int_{B_{2\delta}\cap\R^n_+}| u_{\varepsilon}|^{p}dx}{\left(\int_{B_{\delta}\cap\R^n_+}| u_{\varepsilon}|^{p^{*}}dx\right)^{\frac{p}{p^{*}}}}\\
=& \frac{\int_{B_{2\delta/\varepsilon}\cap\R^n_+}|\wF_{\lambda}(\nabla  U_{\lambda,p})|^{p}dx}{\left(\int_{B_{\delta/\varepsilon}\cap\R^n_+}| U_{\lambda,p}|^{p^{*}}dx\right)^{\frac{p}{p^{*}}}}+C\left(\frac{\varepsilon}{\delta}\right)^p\frac{\int_{B_{2\delta/\varepsilon}\cap\R^n_+}| U_{\lambda,p}|^{p}dx}{\left(\int_{B_{\delta/\varepsilon}\cap\R^n_+}|U_{\lambda,p}|^{p^{*}}dx\right)^{\frac{p}{p^{*}}}}.\\
\end{aligned}\]
Taking the limit as $\varepsilon\rightarrow0$ on both sides, we obtain
\begin{equation}\label{e4.21}A\leq \frac{\int_{\R^n_+} |\wF_{\lambda}(\nabla  U_{\lambda,p})|^{p}dx}{\left(\int_{\R^n_+}|U_{\lambda,p}|^{p^{*}}dx\right)^{\frac{p}{p^{*}}}}=\frac{\int_{\R^n_+} |F_{\lambda}(\nabla  U_{\lambda,p})|^{p}dx}{\left(\int_{\R^n_+}|U_{\lambda,p}|^{p^{*}}dx\right)^{\frac{p}{p^{*}}}}=C^{-p}(\lambda,p).\end{equation}
Combining with \eqref{e4.20} and \eqref{e4.21}, we conclude that $A=C^{-p}(\lambda,p)$,  which accomplishes the proof of Theorem \ref{theorem4.4}.
\end{proof}

From the capillary P\'{o}lya-Szeg\"{o} principle outside convex domain and the Moser Lemma, we derive the following Moser-Trudinger inequality outside  convex domain under a Neumann boundary condition.

\begin{proposition}\label{theorem1.5}
Let  $u\in W_0^{1,n}(\Omega;E^c)$ be a non-negative function such that
 \[\int_{\Omega}|\wF_{\lambda}(\nabla u)|^ndx\leq 1,\]
and satisfying the  anisotropic Neumann boundary condition:
\begin{equation}\label{e5.40}D\wF_{\lambda}(\nabla u)\cdot\nu=0 \mbox{\ on\ } \partial\Omega\cap\partial E^c.\end{equation}
Then,  it follows that
\[\int_{\Omega}\exp\left(\widetilde{\lambda}_n|u|^{\frac{n}{n-1}}\right)dx\leq C_1(n),\]
where $\widetilde{\lambda}_n=n(n\widetilde{\kappa}_n/2)^{\frac{1}{n-1}}$ and $\widetilde{\kappa}_n=|x\in\R^n: F_{\lambda}^o(x)\leq 1|$.
\end{proposition}

\begin{proof}
Since $u$ satisfies the Neumann boundary condition \eqref{e5.40}, the capillary P\'{o}lya-Szeg\"{o} inequality outside convex domain implies that
\[\int_{B_r^+(-r\lambda e_n)}|F_{\lambda}(\nabla u^{*})|^ndx\leq\int_{\Omega}|\wF_{\lambda}(\nabla u)|^ndx\leq1.\]
On the other hand,
\[\begin{aligned}
\int_{B_r^+(-r\lambda e_n)}|u^{*}(F_{\lambda}^o(x))|^ndx=&\int_0^r\int_{\{F_{\lambda}^o(x)=\rho\}\cap\R^n_+}|(u^{*})'(\rho)|^n d\HR^{n-1}d\rho\\
=&\frac{n\widetilde{\kappa}_n}{2}\int_0^r|(u^{*})'(\rho)|^n \rho^{n-1}d\rho,
\end{aligned}\]
which implies that
\[w_{n-1}\int_0^{r}|(u^{*})'(\rho)|^n \rho^{n-1}d\rho\leq\frac{2w_{n-1}}{n\widetilde{\kappa}_n}.\]
From Moser's Lemma, we obtain
\begin{equation}\label{e5.41}
w_{n-1}\int_0^r\exp\left(\widetilde{\lambda}_n |u^{*}(\rho)|^{\frac{n}{n-1}}\rho^{n-1}\right)d\rho\leq C(n),
\end{equation}
where %$\widetilde{\lambda}_n=nw_{n-1}^{\frac{1}{n-1}}\cdot\left(\frac{2w_{n-1}}{n\widetilde{\kappa}_n}\right)^{-\frac{1}{n-1}}=n(n\widetilde{\kappa}_n/2)^{\frac{1}{n-1}}$.
$\widetilde{\lambda}_n=n(n\widetilde{\kappa}_n/2)^{\frac{1}{n-1}}$.
Therefore,
\[\begin{aligned}
\int_{\Omega}\exp\left(\widetilde{\lambda}_n|u|^{\frac{n}{n-1}}\right)dx=&\int_{B_r^+(-r\lambda e_n)}\exp\left(\widetilde{\lambda}_n|u^{*}(F_{\lambda}^o(x))|^{\frac{n}{n-1}}\right)dx\\
=&\int_0^r\int_{\{F_{\lambda}^o(x)=\rho\}\cap\R^n_+}\exp\left(\widetilde{\lambda}_n|u^{*}(\rho)|^{\frac{n}{n-1}}\right)d\HR^{n-1}d\rho\\
=&\frac{n\widetilde{\kappa}_n}{2}\int_0^r\exp\left(\widetilde{\lambda}_n|u^{*}(\rho)|^{\frac{n}{n-1}}\right)\rho^{n-1}d\rho\leq C_1(n),
\end{aligned}\]
where we used   \eqref{e5.41} in the last inequality and $v_n$ denotes the volume of the unit ball. Then, we complete the proof.
\end{proof}

By a subcritical approximation procedure and establishing the subcritical capillary Moser-Trudinger inequality in  $W_{0}^{1,n}(\Omega;E^c)$, the Neumann boundary condition \eqref{e5.40} of Proposition \ref{theorem1.5} can be removed.

\begin{proof}[\bf{Proof of the Theorem \ref{theorem1.6}}]
 The proof is achieved   through several steps.

\medskip
Step 1. We first claim that
\[\sup_{\int_{\Omega}|\wF_{\lambda} (\nabla u)|^ndx\leq1}\exp\left((\widetilde{\lambda}_n-\varepsilon_k)|u|^{\frac{n}{n-1}}\right)dx<\infty\]
for any $u\in W_{0}^{1,n}(\Omega;E^c)$ and $\varepsilon_k$ small.

\medskip
Assume by contradiction that there exists a sequence  $\{u_k\}\in W_{0}^{1,n}(\Omega;E^c)$ with $\int_{\Omega}|\wF_{\lambda} (\nabla u_k)|^ndx\leq1$ such that
\begin{equation}\label{e6.2}\int_{\Omega}\exp\left((\widetilde{\lambda}_n-\varepsilon_k)|u_k|^{\frac{n}{n-1}}\right)dx=\infty.\end{equation}

If $u_k\rightharpoonup 0$ weakly, we take a non-negative cut-off  function $0\leq\phi(x)\leq1$ belongs to $C_c^{\infty}(B_{\delta}(x_0))$ for any $x_0\in\Omega$, with $\delta$ sufficiently small so that $\phi(x)=1\mbox{\ in\ }B_{\frac{\delta}{2}}(x_0)$. Define
\[\wF_{\lambda}(\xi,x_0)=|\xi|-\nabla h(x_0)\cdot\xi.\]
Then we have
\[\begin{split}
\int_{B_{\delta}(x_0)}\left|\wF_{\lambda}\left(\nabla(u_k \phi)\right)\right|^ndx
=&\int_{B_{\delta}(x_0)}\left|\wF_{\lambda}\left(\phi\nabla u_k +u_k\nabla\phi\right)\right|^ndx\\
=&\int_{B_{\delta}(x_0)}\left|\wF_{\lambda}\left(\phi\nabla u_k +u_k\nabla\phi,x_0\right)\right|^ndx+o(\delta)\\
\leq& \int_{B_{\delta}(x_0)}\left|\wF_{\lambda}\left(\nabla u_k,x_0\right)\right|^n \phi^ndx+\int_{B_{\delta}(x_0)}\left|\wF_{\lambda}\left(\nabla\phi,x_0\right)\right|^n|u_k|^ndx+o(\delta)\\
\leq&1+\eta
\end{split}\]
for $\eta$ small. Setting
\[\left|\{\wF_{\lambda}^o(\xi,x_0)\leq1\}\right|=\kappa_{n,x_0},\]
 using a normalization argument, the anisotropic Moser-Trudinger inequality yields
\[\int_{B_{\frac{\delta}{2}}(x_0)}\exp\left(\frac{\hat{\lambda}_n}{(1+\eta)^{\frac{1}{n-1}}}|u_k|^{\frac{n}{n-1}}\right)dx<\infty,\]
where $\hat{\lambda}_n=n\left({n\hat{\kappa}_{n,x_0}}\right)^{\frac{1}{n-1}}$. Since \[\frac{\hat{\lambda}_n}{(1+\eta)^{\frac{1}{n-1}}}>\widetilde{\lambda}_n-\varepsilon_k,\]
the arbitrariness of $x_0$ implies that
\[\int_{\Omega}\exp\left((\widetilde{\lambda}_n-\varepsilon_k)|u_k|^{\frac{n}{n-1}}\right)dx
<\int_{\Omega}\exp\left(\frac{\hat{\lambda}_n}{(1+\eta)^{\frac{1}{n-1}}}|u_k|^{\frac{n}{n-1}}\right)dx<\infty,\]
which is a contradiction with \eqref{e6.2}.

\medskip
For any $x_0 \in \partial \Omega$, we extend $u_k$ to be an even function in $B_\delta(x_0)$. Arguing as before, we obtain
\[\int_{B_{\delta}(x_0)}\left|\wF_{\lambda}\left(\nabla(u_k \phi)\right)\right|^ndx\leq2(1+\eta),\]
which yields
\[\int_{B_{\frac{\delta}{2}}(x_0)}\exp\left(\frac{\hat{\lambda}_n}{(2(1+\eta))^{\frac{1}{n-1}}}|u_k|^{\frac{n}{n-1}}\right)dx<\infty.\]
This again results in a contradiction.

\medskip
Now consider the case where $u_k\rightharpoonup u$ weakly with $u\not\equiv0$,
\[\begin{split}
\int_{\Omega}\exp\left((\widetilde{\lambda}_n-\varepsilon_k)|u_k|^{\frac{n}{n-1}}\right)dx\leq&\int_{\Omega}\exp
\left((\widetilde{\lambda}_n-\varepsilon_k)(1+\sigma)|u_k-u|^{\frac{n}{n-1}}+(\widetilde{\lambda}_n-\varepsilon_k)C_{\sigma}|u|^{\frac{n}{n-1}}\right)dx\\
\leq&\left(\int_{\Omega}\exp\left((\widetilde{\lambda}_n-\varepsilon_k)(1+\sigma)p|u_k-u|^{\frac{n}{n-1}}\right)dx\right)^{\frac{1}{p}}
\left(\int_{\Omega}\exp\left((\widetilde{\lambda}_n-\varepsilon_k)C_{\sigma}q|u|^{\frac{n}{n-1}}\right)dx\right)^{\frac{1}{q}},
\end{split}\]
where we choose $p>1$ sufficiently close to 1  such that  $p$ and $q$ are conjugate indices. Consequently, \eqref{e6.2} yields that
\begin{equation}\label{e6.3}
\lim_{k\rightarrow\infty}\int_{\Omega}\exp\left((\widetilde{\lambda}_n-\varepsilon_k)(1+\sigma)p|u_k-u|^{\frac{n}{n-1}}\right)dx=\infty.
\end{equation}

Taking the same cut-off function $\phi$ as before, we estimate
\[\begin{split}
\int_{B_{\delta}(x_0)}\left|\wF_{\lambda}\left(\nabla(u_k-u)\right)\right|^ndx
\leq\int_{B_{\delta}(x_0)}\left|\wF_{\lambda}\left(\nabla u_k \right)\right|^ndx+\int_{B_{\delta}(x_0)}\left|\wF_{\lambda}\left(\nabla u\right)\right|^ndx
\leq1+\eta
\end{split}\]
for $\eta$ sufficiently small. Moreover,
\[\begin{split}
\int_{B_{\delta}(x_0)}\left|\wF_{\lambda}\left(\nabla((u_k-u)\phi)\right)\right|^ndx
=&\int_{B_{\delta}(x_0)}\left|\wF_{\lambda}\left(\phi\nabla (u_k-u) +(u_k-u)\nabla\phi\right)\right|^ndx\\
=&\int_{B_{\delta}(x_0)}\left|\wF_{\lambda}\left(\phi\nabla (u_k-u) +(u_k-u)\nabla\phi,x_0\right)\right|^ndx+o(\delta)\\
\leq& \int_{B_{\delta}(x_0)}\left|\wF_{\lambda}\left(\nabla (u_k-u),x_0\right)\right|^n \phi(x)^ndx+\int_{B_{\delta}(x_0)}\left|\wF_{\lambda}\left(\nabla\phi,x_0\right)\right|^n|u_k-u|^ndx+o(\delta)\\
\leq&1+2\eta.
\end{split}\]
Then, applying the anisotropic Moser-Trudinger inequality, we obtain
\[\int_{B_{\frac{\delta}{2}}(x_0)}\exp\left(\frac{\hat{\lambda}_n}{(1+2\eta)^{\frac{1}{n-1}}}|u_k-u|^{\frac{n}{n-1}}\right)dx<\infty,\]
where $\hat{\lambda}_n=n\left({n\hat{\kappa}_{n,x_0}}\right)^{\frac{1}{n-1}}$. Now, choose $p > 1$ sufficiently close to 1  and $\sigma > 0$ small enough, we have
 \[\frac{\hat{\lambda}_n}{(1+2\eta)^{\frac{1}{n-1}}}>(\widetilde{\lambda}_n-\varepsilon_k)(1+\delta)p\quad\text{for all large }k.\]
 This inequality contradicts the \eqref{e6.3}.

\medskip
The argument for boundary points $x_0 \in \partial\Omega$ is analogous. We extend $u_k$ and $u$ to $B_\delta(x_0)$ as described previously. Using the same cut-off function $\phi$, we then obtain the energy estimate
\[\int_{B_{\delta}(x_0)}\left|\wF_{\lambda}\left(\nabla((u_k-u) \phi)\right)\right|^ndx\leq2(1+2\eta).\]
Applying the anisotropic Moser-Trudinger inequality again, we have
\[\int_{B_{\frac{\delta}{2}}(x_0)}\exp\left(\frac{\hat{\lambda}_n}{(2(1+2\eta))^{\frac{1}{n-1}}}|u_k|^{\frac{n}{n-1}}\right)dx<\infty.\]
However, with $p>1$ chosen sufficiently close to $1$ and $\sigma>0$ small, we again have
\[\frac{\hat{\lambda}_n}{2(1+2\eta)^{\frac{1}{n-1}}}>(\widetilde{\lambda}_n-\varepsilon_k)(1+\delta)p\]
for all large $k$. Using again the arbitrariness of $x_0$, we obtain
\[\int_{\Omega}\exp\left((\widetilde{\lambda}_n-\varepsilon_k)(1+\sigma)p|u_k-u|^{\frac{n}{n-1}}\right)dx
<\int_{\Omega}\exp\left(\frac{\hat{\lambda}_n}{(2(1+2\eta))^{\frac{1}{n-1}}}|u_k|^{\frac{n}{n-1}}\right)dx<\infty,\]
which contradicts with \eqref{e6.3}.

\medskip
Step 2. For any $u\in W_0^n(\Omega;E^c)$, define
\[I_{\varepsilon}=\sup_{\int_{\Omega}\left|\wF_{\lambda}(\nabla u)\right|^ndx\leq 1}\int_{\Omega}\exp\left((\widetilde{\lambda}_n-\varepsilon)|u|^{\frac{n}{n-1}}\right)dx.\]
We claim that for every $\varepsilon>0$, $I_{\varepsilon}$ is attained by some function $u_{\varepsilon}$.

\medskip
Let $\{u_j\}$ be a maximizing sequence for $I_{\varepsilon}$, i.e.,
\[\int_{\Omega}\left|\wF_{\lambda}(\nabla u_j)\right|^ndx\leq 1,\quad u_j=0\text{ on }\partial\Omega\cap E^c\]
and
\[\lim_{j\rightarrow\infty}\int_{\Omega}\exp\left((\widetilde{\lambda}_n-\varepsilon)|u_j|^{\frac{n}{n-1}}\right)= I_{\varepsilon}.\]
Since the sequence ${u_j}$ is bounded in $W_0^n(\Omega; E^c)$, there exists a function $u_\varepsilon \in W_0^n(\Omega; E^c)$ and a subsequence (still denoted by ${u_j}$) such that
\[u_j\rightharpoonup u_{\varepsilon} \text{ weakly in }W_0^n(\Omega; E^c)\quad \text{and}\quad u_j\rightarrow u_{\varepsilon} \text{ a.e. in } \Omega.\]
Applying the result of Step 1, we obtain
\[\int_{\Omega}\exp\left(\widetilde{\lambda}_n-\frac{\varepsilon}{2}|u_j|^{\frac{n}{n-1}}\right)dx<\infty.\]
Consequently, the sequence $\left\{\exp\left(\widetilde{\lambda}_n-\frac{\varepsilon}{2}|u_j|^{\frac{n}{n-1}}\right)\right\}$  is uniformly integrable. Together with the almost everywhere convergence, Vitali's convergence theorem implies
\[\lim_{j\rightarrow\infty}\int_{\Omega}\exp\left(\widetilde{\lambda}_n-\varepsilon|u_j|^{\frac{n}{n-1}}\right)dx
=\int_{\Omega}\exp\left(\widetilde{\lambda}_n-\varepsilon|u_{\varepsilon}|^{\frac{n}{n-1}}\right)dx.\]
Therefore, $I_\varepsilon=\int_{\Omega}\exp\left((\widetilde{\lambda}_n-\varepsilon)|u_{\varepsilon}|^{\frac{n}{n-1}}\right)dx$, and $u_\varepsilon$ is indeed an extremal function for $I_\varepsilon$.

\medskip
Step 3. We now verify that the extremal function $u_{\varepsilon}$ automatically satisfies the Neumann boundary condition.
%By the standard variational method, $u_{\varepsilon}$ is the critical point of the following functional $J(u)$ for any $\lambda\in\R$:
%\[J(u)=\frac{n}{n-1}\int_{\Omega}|u|^{\frac{n}{n-1}}\exp\left((\widetilde{\lambda}_n-\varepsilon)|u|^{\frac{n}{n-1}}\right)dx
%+n\lambda\int_{\Omega}\div\left(|\wF_{\lambda}(\nabla u)|^{n-1} D\wF_{\lambda}(\nabla u)\right)udx.\]
Define
\[I(u)=\int_{\Omega}\exp\left(\widetilde{\lambda}_n-\varepsilon|u|^{\frac{n}{n-1}}\right)dx\quad\text{and}\quad G(u)=\int_{\Omega}\left|\wF_{\lambda}(\nabla u)\right|^ndx.\]
By standard variational arguments, $u_{\varepsilon}$ satisfies the Euler-Lagrange equation associated with the constrained maximization problem, namely,
\[\left<I'(u_{\varepsilon}),\phi\right>=\lambda \left<G'(u_{\varepsilon}),\phi\right>\quad\text{for any } \phi\in C_0^{\infty}(\Omega) \]
where $\lambda \in \mathbb{R}$ is a Lagrange multiplier. Explicitly, this reads
%By the standard variational method, $u_{\varepsilon}$ satisfies the Euler-Lagrange equation $\left<I'(u),\phi\right>=\left<G'(u),\phi\right>$ for any $\phi\in C_0^{\infty}(\Omega)$ and $\lambda\in\R$, that is
%\[I'(u)=\frac{n}{n-1}|u|^{\frac{1}{n-1}}\exp\left((\widetilde{\lambda}_n-\varepsilon)|u|^{\frac{n}{n-1}}\right)\quad\text{and}\quad G'(u)=-n\div\left(|\wF_{\lambda}(\nabla u)|^{n-1} D\wF_{\lambda}(\nabla u)\right).\]
%\[\frac{n}{n-1}|u_{\varepsilon}|^{\frac{1}{n-1}}\exp\left((\widetilde{\lambda}_n-\varepsilon)|u_{\varepsilon}|^{\frac{n}{n-1}}\right)=-n\lambda\div\left(|\wF_{\lambda}(\nabla u_{\varepsilon})|^{n-1} D\wF_{\lambda}(\nabla u_{\varepsilon})\right).\]
\begin{equation}\label{e6.4}\begin{split}&\int_{\Omega}\frac{n}{n-1}|u_{\varepsilon}|^{\frac{1}{n-1}}\exp\left((\widetilde{\lambda}_n-\varepsilon)|u_{\varepsilon}|^{\frac{n}{n-1}}\right)\phi(x)dx\\
=&-n\lambda\int_{\Omega}\div\left(|\wF_{\lambda}(\nabla u_{\varepsilon})|^{n-1} D\wF_{\lambda}(\nabla u_{\varepsilon})\right)\phi(x)dx\\
=&n\lambda\int_{\Omega}\left|\wF_{\lambda}(\nabla u_{\varepsilon})\right|^{n-1}D\wF_{\lambda}(\nabla u_{\varepsilon})\cdot D\phi dx.
\end{split}\end{equation}
Now we take the test function $\varphi\in C^{\infty}(\Omega)$ satisfying $\varphi=0$ on $\partial\Omega\cap E^c$, then we have
\begin{equation}\label{e6.5}\begin{split}&\int_{\Omega}\frac{n}{n-1}|u_{\varepsilon}|^{\frac{1}{n-1}}\exp\left((\widetilde{\lambda}_n-\varepsilon)|u_{\varepsilon}|^{\frac{n}{n-1}}\right)\varphi(x)dx\\
=&-n\lambda\int_{\Omega}\div\left(|\wF_{\lambda}(\nabla u_{\varepsilon})|^{n-1} D\wF_{\lambda}(\nabla u_{\varepsilon})\right)\varphi(x)dx\\
=&n\lambda\int_{\Omega}\left|\wF_{\lambda}(\nabla u_{\varepsilon})\right|^{n-1}D\wF_{\lambda}(\nabla u_{\varepsilon})\cdot D\varphi dx-n\lambda\int_{\partial\Omega\cap\partial E^c}\left|\wF_{\lambda}(\nabla u_{\varepsilon})\right|^{n-1}D\wF_{\lambda}(\nabla u_{\varepsilon})\cdot\nu\varphi dx.\\
\end{split}\end{equation}
Combining \eqref{e6.4} and \eqref{e6.5}, we deduce that
%\[n\lambda\int_{\partial\Omega\cap\partial E^c}\left|\wF_{\lambda}(\nabla u_{\varepsilon})\right|^{n-1}D\wF_{\lambda}(\nabla u_{\varepsilon})\cdot\nu\varphi dx=0\]
%for any $\lambda\in\R$ and $\varphi\in C^{\infty}(\Omega)$ with $\varphi=0$ on $\partial\Omega\cap E^c$.
\[D\wF_{\lambda}(\nabla u_{\varepsilon})\cdot\nu=0\quad\text{on}\quad \partial\Omega\cap\partial E^c.\]

\medskip
Step 4.
%For any $u\in W_0^{1,n}(\Omega;E^c)$ with $\int_{\Omega}|\wF_{\lambda}(\nabla u)|^ndx=1$,
%\[\begin{split}
%\int_{\Omega}\exp\left(\widetilde{\lambda}_n|u|^{\frac{n}{n-1}}\right)dx=&\lim_{\varepsilon\rightarrow0}\int_{\Omega}\exp\left((\widetilde{\lambda}_n-\varepsilon)|u|^{\frac{n}{n-1}}\right)dx\\
%\leq&\lim_{\varepsilon\rightarrow0}\int_{\Omega}\exp\left((\widetilde{\lambda}_n-\varepsilon)|u|^{\frac{n}{n-1}}\right)dx<\infty.
%\end{split}\]
We have shown that in the subcritical case,
%\[\sup_{\int_{\Omega}|\wF_{\lambda} (\nabla u)|^ndx\leq1}\exp\left((\widetilde{\lambda}_n-\varepsilon_k)|u|^{\frac{n}{n-1}}\right)dx\]
$I_{\varepsilon}$ is attained by an extremal function $u_{\varepsilon}$,  which automatically satisfies the Neumann boundary condition \eqref{e5.40}. To approximate the critical case, we first claim that
\begin{equation}\label{e6.1}\lim_{\varepsilon\rightarrow0}\int_{\Omega}\exp\left((\widetilde{\lambda}_n-\varepsilon)|u_{\varepsilon}|^{\frac{n}{n-1}}\right)dx=\sup_{\int_{\Omega}|\wF_{\lambda} (\nabla u)|^ndx\leq1}\int_{\Omega}\exp\left(\widetilde{\lambda}_n|u|^{\frac{n}{n-1}}\right)dx.\end{equation}
  Indeed, for any function $u$ satisfies $\int_{\Omega}|\wF_{\lambda} (\nabla u)|^ndx\leq1$, we have
\[\begin{aligned}
\int_{\Omega}\exp\left(\widetilde{\lambda}_n|u|^{\frac{n}{n-1}}\right)dx
=&\lim_{\varepsilon\rightarrow0}\int_{\Omega}\exp\left((\widetilde{\lambda}_n-\varepsilon)|u|^{\frac{n}{n-1}}\right)dx\\
\leq&\lim_{\varepsilon\rightarrow0}\int_{\Omega}\exp\left((\widetilde{\lambda}_n-\varepsilon)|u_{\varepsilon}|^{\frac{n}{n-1}}\right)dx.
\end{aligned}\]
Taking the supremum over all such $u$ on the left-hand side yields
\[\sup_{\int_{\Omega}|\wF_{\lambda} (\nabla u)|^ndx\leq1}\int_{\Omega}\exp\left(\widetilde{\lambda}_n|u|^{\frac{n}{n-1}}\right)dx\leq\lim_{\varepsilon\rightarrow0}\int_{\Omega}\exp\left((\widetilde{\lambda}_n-\varepsilon)|u_{\varepsilon}|^{\frac{n}{n-1}}\right)dx.\]
Since the reverse inequality follows directly from the definition of $I_\varepsilon$, we obtain \eqref{e6.1}.

\medskip
On the other hand, since each $u_{\varepsilon}$ satisfies the Neumann boundary condition, we may apply Proposition  \ref{theorem1.5} to obtain
\[\int_{\Omega}\exp\left(\widetilde{\lambda}_n|u_{\varepsilon}|^{\frac{n}{n-1}}\right)dx<\infty,\]
where $\widetilde{\lambda}_n=n(n\widetilde{\kappa}_n/2)^{\frac{1}{n-1}}$ and $\widetilde{\kappa}_n=|x\in\R^n: F_{\lambda}^o(x)\leq 1|$.
Consequently,
\[\int_{\Omega}\exp\left(\widetilde{\lambda}_n|u|^{\frac{n}{n-1}}\right)dx
=\lim_{\varepsilon\rightarrow0}\int_{\Omega}\exp\left((\widetilde{\lambda}_n-\varepsilon)|u_{\varepsilon}|^{\frac{n}{n-1}}\right)dx
\leq \int_{\Omega}\exp\left(\widetilde{\lambda}_n|u_{\varepsilon}|^{\frac{n}{n-1}}\right)dx<\infty.\]
\end{proof} Then we accomplish the proof of Theorem \ref{theorem1.6}.

\section{Proof of the Theorem \ref{theorem4.2},  Theorem \ref{corollary4.1} }
  In this  section, we establish a capillary Talenti comparison principle and a capillary Bossel-Daners inequality outside convex domain.  We begin with the following two lemmas.

\bl\label{lemma3.1}
Let $u\in L^{p}(\Omega)$ for $1\leq p\leq\infty$. Under the same notations and hypotheses as in Theorem \ref{theorem4.2}, we get
\be\label{e5.16}
\|u\|_{L^p(\Omega)}\leq\|v\|_{L^p(B^+_r(-r\lambda e_n))}\quad\mbox{for all\ }1\leq p\leq\infty,
\ee
and
\be\label{e5.17}
\|\wF_{\lambda}(\nabla u)\|_{L^q(\Omega)}\leq\|F_{\lambda}(\nabla v)\|_{L^q(B^+_r(-r\lambda e_n))}\quad\mbox{for \ }1\leq q\leq 2.
\ee
\el
\bp
If $p=\infty$, then the result is contained in the definition of the capillary Schwarz rearrangement. Assume that $1\leq p<\infty$. Since $u$ and $u^{*}$ are equimeasurable, i.e.
\[|\{u>t\}|=|\{u^{*}>t\}|,\]
then by the Fubini's theorem, $u$ and $u^{*}$ have the same $L^p$ norms.  Applying the capillary Talenti comparison principle, we have
\[\int_{\Omega}|u|^pdx=\int_{B^+_r(-r\lambda e_n)}|u^{*}|^pdx\leq \int_{B^+_r(-r\lambda e_n)}|v|^pdx.\]

\noindent Moreover, let $1\leq q\leq 2$ and $t>0$. By H\"{o}lder's inequality, we have
\[\frac{1}{h}\int_{\{t<u\leq t+h\}}|\wF_{\lambda}(\nabla u)|^q dx\leq\left(\frac{1}{h}\int_{\{t<u\leq t+h\}}|\wF_{\lambda}(\nabla u)|^2 dx\right)^{\frac{q}{2}}\left(\frac{1}{h}\int_{\{t<u\leq t+h\}}dx\right)^{\frac{2-q}{2}}.\]
Letting $h\rightarrow 0$ and using \eqref{e4.5}, we obtain
\[\begin{split}
-\frac{d}{dt}\int_{\{u>t\}}|\wF_{\lambda}(\nabla u)|^q dx\leq&\left(-\frac{d}{dt}\int_{\{u>t\}}|\wF_{\lambda}(\nabla u)|^2 dx\right)^{\frac{q}{2}}(-\mu'(t))^{\frac{2-q}{2}}\\
\leq&G(\mu(t))^{\frac{q}{2}}(-\mu'(t))^{\frac{2-q}{2}}.
\end{split}\]
Assume that $M$ is the maximum value of $u$ in $\Omega$. Integrating the above inequality from $0$ to $M$, we have
\[\begin{split}
\int_{\Omega}|\wF_{\lambda}(\nabla u)|^q dx\leq&\int_0^{M}[G(\mu(t))(-\mu'(t))^{-1}]^{\frac{q}{2}}(-\mu'(t))dt\\
\leq&\int_0^{M}\left[(n\kappa_{\lambda}^{\frac{1}{n}})^{-2}\mu(t)^{\frac{2}{n}-2}G^2(\mu(t))\right]^{\frac{q}{2}}(-\mu'(t))dt,
\end{split}\]
where we used the \eqref{e4.7} in the last inequality. Therefore, by the change of variable $\xi=\mu(t)$, we obtain
\be\label{e5.19}
\int_{\Omega}|\wF_{\lambda}(\nabla u)|^q dx\leq\int_0^{|\Omega|}\left[(n\kappa_{\lambda}^{\frac{1}{n}})^{-1}\xi^{\frac{1}{n}-1}G(\xi)\right]^q d\xi.
\ee
Moreover, by introducing polar coordinates and employing the explicit form of  $v$ given in \eqref{e5.18}, we derive
\[\begin{split}\int_{B^+_r(-r\lambda e_n)}|F_{\lambda}(\nabla v)|^qdx=&\int_0^r|v'(\rho)|^qn\kappa_{\lambda}\rho^{n-1}d\rho\\
=&\int_0^r|(n\kappa_{\lambda})^{-1}\rho^{1-n}G(\kappa_{\lambda}\rho^n)|^q n\kappa_{\lambda}\rho^{n-1}d\rho.\end{split}\]
Setting $\zeta=\kappa_{\lambda}\rho^n$ yields
\[\begin{split}
\int_{B^+_r(-r\lambda e_n)}|F_{\lambda}(\nabla v)|^qdx=&\int_{0}^{\kappa_{\lambda}r^n}\left[(n\kappa_{\lambda})^{-1}(\zeta\kappa_{\lambda}^{-1})^{\frac{1-n}{n}}G(\zeta)\right]^qd\zeta\\
=&\int_0^{|\Omega|}\left[(n\kappa_{\lambda}^{\frac{1}{n}})^{-1}\zeta^{\frac{1}{n}-1}G(\zeta)\right]^qd\zeta,
\end{split}\]
where the right-hand side coincides with that of \eqref{e5.19}. This completes the proof.
\ep

\bl\label{lemma3.2}
Let $f>0$. Equality holds in either inequality \eqref{e5.16} or \eqref{e5.17} if and only if $u^{*}=v$ almost everywhere. Thus, either we have equality in all these inequalities or all of these inequalities are strict.
\el
\bp
Step 1. For $1\leq p<\infty$,  $u$ and $u^{*}$ have the same $L^p$ norms  and satisfy the inequality $0\leq u^{*}\leq v$. Consequently, the norm equality $\|u\|_{L^p(\Omega)}=\|v\|_{L^p(B^+_r(-r\lambda e_n))}$ implies that $u^{*}=v$ almost everywhere. And the converse is trivially true in this case.

\medskip
Step 2. For $p=\infty$.
Assume that  $\mu$ and $\nu$ represent the distribution function of $u$ and $v$ respectively. Since $u^{*}\leq v$, then
\[\mu(t)=|\{u^{*}>t\}|\leq|\{v>t\}|=\nu(t)\]
for any $t>0$. From \eqref{e4.10}, we have
\[t=\left(n\kappa_{\lambda}^{\frac{1}{n}}\right)^{-2}\int_{\nu(t)}^{|\Omega|}\xi^{\frac{2}{n}-2}G(\xi)d\xi.\]
Differentiating both sides with respect to t simultaneously, we have
\[1=\left(n\kappa_{\lambda}^{\frac{1}{n}}\right)^{-2}(\nu(t))^{\frac{2}{n}-2}G(\nu(t))(-\nu'(t)).\]
Combining this with \eqref{e4.7}, we get
\[K(\mu(t))\mu'(t)\leq K(\nu(t))\nu'(t),\]
where $K(t)=\left(n\kappa_{\lambda}^{\frac{1}{n}}\right)^{-2}t^{\frac{2}{n}-2}$. If $\widetilde{K}$ is a primitive of $K$ and set $\eta=\widetilde{K}(\nu(t))-\widetilde{K}(\mu(t))$, then the above inequality can be rewritten as
\[\frac{d}{dt}(\eta(t))\geq 0.\]
 Set $M=\|u\|_{L^{\infty}(\Omega)}=\|v\|_{L^{\infty}(B^+_r(-r\lambda e_n))}$. We have established that the function $\eta(t)$ is non-increasing on the interval  $[0,M]$. Moreover, since $\eta(0)=\eta(M)=0$, it follows that $\eta$ vanishes identically on $[0,M]$. The non-negativity of $K$ ensures that $\widetilde{K}$ is strictly increasing. Indeed, if  $\widetilde{K}$ were not strictly increasing, it would follow that $G\equiv0$  and consequently $f=0$, which is a contradiction with our assumption. Therefore, the identity $\eta\equiv0$ indicates that $\mu(t)=\nu(t)$ for any $t\in[0,M]$. As a consequence,  $u^{*}$ and $v$ are equimeasurable, and so $u^{*}=v$ almost everywhere. The converse implication follows immediately from the definitions.

\medskip
Step 3. For $1\leq q\leq 2$. Assume that $\|\wF_{\lambda}(\nabla u)\|_{L^q(\Omega)}\leq\|F_{\lambda}(\nabla v)\|_{L^q(B^+_r(-r\lambda e_n))}$. Getting back to the proof of Lemma \ref{lemma3.2}, all the inequalities hold as equalities, and then
\[1=\left(n\kappa_{\lambda}^{\frac{1}{n}}\right)^{-2}(\mu(t))^{\frac{2}{n}-2}G(\mu(t))(-\mu'(t)).\]
This identity immediately implies that $\eta'(t)=0$ for all $t\in[0,M]$. Repeating the similar proof from the previous step, we conclude that $u^{*}=v$ almost everywhere.

\noindent Conversely, suppose $u^{*}=v$ holds. Combining this equality with the P\'{o}lya-Szeg\"{o} principle of the convex symmetrization outside a convex cylinder $E$ (cf. Theorem \ref{theorem4.1}) and the inequality \eqref{e5.17}, we deduce that
\[\|F_{\lambda}(\nabla v)\|_{L^q(B^+_r(-r\lambda e_n))}=\|F_{\lambda}(\nabla u^{*})\|_{L^q(B^+_r(-r\lambda e_n))}\leq \|\wF_{\lambda}(\nabla u)\|_{L^q(\Omega)}\leq \|F_{\lambda}(\nabla v)\|_{L^q(B^+_r(-r\lambda e_n))},\]
which indicates that all of the above are considered as equalities. Then we complete the proof.
\ep

%Now, we can prove the Talenti comparison principle.
\begin{proof}[{\bf Proof of the Theorem \ref{theorem4.2}}]
Step 1.
A function  $u\in W_0^{1,2}(\Omega; E^c)$ is  a weak solution of \eqref{e4.2} if it satisfies  the variational identity
\be\label{e4.3}
\int_{\Omega}\wF_{\lambda}(\nabla u)D\wF_{\lambda}(\nabla u)\cdot\nabla\phi dx-\int_{\partial\Omega\cap E^c} \wF_{\lambda}(\nabla u)D\wF_{\lambda}(\nabla u)\cdot\nu\phi dx=\int_{\Omega}f\phi dx,
\ee
for any $\phi\in W_0^{1,2}(\Omega; E^c)$.
For $t>0$, let $\phi=(u-t)^+$ and substituting it into the weak formulation \eqref{e4.3}  yields
\[\int_{\{u>t\}}\wF_{\lambda}^2(\nabla u)dx=\int_{\{u>t\}}f(u-t)dx.\]
Noting that $\int_{\{u>t\}}\wF_{\lambda}^2(\nabla u)dx$ is a decreasing function of $t$, we differentiate with respect to $t$ to obtain
\be\label{e4.5}
0\leq -\frac{d}{dt}\int_{\{u>t\}}\wF_{\lambda}^2(\nabla u)dx=\int_{\{u>t\}}fdx.\ee
Applying the Cauchy-Schwartz inequality, we have
\[\left(\frac{1}{h}\int_{\{t<u\leq t+h\}}\wF_{\lambda}(\nabla u)dx\right)^2\leq\left(\frac{1}{h}\int_{\{t<u\leq t+h\}}\wF^2_{\lambda}(\nabla u)dx\right)\left(\frac{1}{h}\int_{\{t<u\leq t+h\}}dx\right).\]
Taking the limit as $h\rightarrow 0$ and using \eqref{e4.5}, we derive
\be\label{e4.6}
\left(-\frac{d}{dt}\int_{\{u>t\}}\wF_{\lambda}(\nabla u)dx\right)^2\leq-\mu'(t)\int_{\{u>t\}}fdx,
\ee
where $\mu(t)$ is the distribution function of $u$. Define $G(\xi)=\int_0^{\xi}f^{\#}(s)ds$,
where $f^{\#}(s)=\inf\{t:|f(x)>t|<s,\ s>0\}$  is the decreasing rearrangement of $f$. Since
\[0\leq\int_{\{u>t\}}fdx\leq \int_0^{\mu(t)}f^{\#}(s)ds=G(\mu(t)),\]
it follows that  $G(\mu(t))$ is non-negative for all $t>0$. Applying Theorem \ref{theorem4.5} with $p=1$, we obtain
\[
-\frac{d}{dt}\int_{\{u>t\}}\wF_{\lambda}(\nabla u)dx=\int_{\{u=t\}}\wF_{\lambda}\left(\frac{\nabla u}{|\nabla u|}\right)d\HR^{n-1}=P_{\wF_{\lambda}}(\{u>t\};E^c).
\]
Using the capillary isomertic inequality \eqref{e1.4}, we get
\[-\frac{d}{dt}\int_{\{u>t\}}\wF_{\lambda}(\nabla u)dx\geq n\kappa_{\lambda}^{\frac{1}{n}}\mu(t)^{1-\frac{1}{n}},\]
 where $\kappa_{\lambda}=|B^+(-\lambda e_n)|$. Combining these results, we arrive at
\[\left(n\kappa_{\lambda}^{\frac{1}{n}}\right)^2\mu(t)^{2-\frac{2}{n}}\leq-\mu'(t) G(\mu(t)).\]
Rewriting this inequality, we have
\be\label{e4.7}
1\leq\left(n\kappa_{\lambda}^{\frac{1}{n}}\right)^{-2}\mu(t)^{\frac{2}{n}-2}\left(-\mu'(t)\right) G(\mu(t)).
\ee
Integrating  from $t'$ to $t$ and making the change of variables $\xi=\mu(t)$, we obtain
\be\label{e4.8}
t-t'\leq\left( n\kappa_{\lambda}^{\frac{1}{n}}\right)^{-2}\int_{\mu(t)}^{\mu(t')}\xi^{\frac{2}{n}-2}G(\xi)d\xi.\ee
Setting $t'=0$ and $t=u^{\#}(s)-\eta$ for $\eta>0$ and $s\in(0,|\Omega|)$, we have $\mu(t')=|\Omega|$ and
$\mu(t)=|\{u>u^{\#}(s)-\eta\}|\geq s$. Since the integrand in \eqref{e4.8} is non-negative, it follows that
\[u^{\#}(s)-\eta\leq \left(n\kappa_{\lambda}^{\frac{1}{n}}\right)^{-2}\int_{s}^{|\Omega|}\xi^{\frac{2}{n}-2}G(\xi)d\xi.\]
By the arbitrariness of $\eta>0$, we conclude that
\be\label{e4.9}
u^{\#}(s)\leq \left(n\kappa_{\lambda}^{\frac{1}{n}}\right)^{-2}\int_{s}^{|\Omega|}\xi^{\frac{2}{n}-2}G(\xi)d\xi.
\ee

Step 2. We will show that the solution of equation \eqref{e4.4} is also radially symmetric. We say that a function $v\in W_0^{1,p}(B^+_r(-r\lambda e_n);\R^n_+)$ is a weak solution to \eqref{e4.4} if
\be\label{e5.14}
\int_{B^+_r(-r\lambda e_n)}F_{\lambda}(\nabla v)DF_{\lambda}(\nabla v)\cdot D\phi dx-\int_{\partial B_r(-r\lambda e_n)\cap\R^n_+}F_{\lambda}(\nabla v)DF_{\lambda}(\nabla v)\cdot\nu\phi d\HR^{n-1}=\int_{B^+_r(-r\lambda e_n)}f^{*}\phi dx,
\ee
for any $\phi\in W_0^{1,p}(B_r(-r\lambda e_n);\R^n_{+})$, and the corresponding energy function
\[J(v)=\frac{1}{2}\int_{B^+_r(-r\lambda e_n)}|F_{\lambda}(\nabla v)|^2 dx-\int_{B^+_r(-r\lambda e_n)}f^{*}v dx.\]
For any $w\in W_0^{1,p}(B_{r}(-r\lambda e_n);\R^n_{+})$ with $F_{\lambda}(\nabla w)DF_{\lambda}(\nabla w)\cdot\nu=0$ on $B_r(-r\lambda e_n)\cap\partial\R^n_+$, we choose the test function $\phi=v-w$ in \eqref{e5.14} and then obtain
\[\int_{B^+_r(-r\lambda e_n)}F_{\lambda}(\nabla v)DF_{\lambda}(\nabla v)\cdot D(v-w) dx=\int_{B^+_r(-r\lambda e_n)}f^{*}(v-w) dx.\]
Obviously, the solution $v$ of problem \eqref{e4.4} is a global minimizer of $J(u)$. Indeed,
\[\begin{split}
J(v)-J(w)=&\frac{1}{2}\int_{B^+_r(-r\lambda e_n)}|F_{\lambda}(\nabla v)|^2 dx-\frac{1}{2}\int_{B^+_r(-r\lambda e_n)}|F_{\lambda}(\nabla w)|^2 dx-\int_{B^+_r(-r\lambda e_n)}f^{*}(v-w) dx\\
=&\frac{1}{2}\int_{B^+_r(-r\lambda e_n)}|F_{\lambda}(\nabla v)|^2 dx-\frac{1}{2}\int_{B^+_r(-r\lambda e_n)}|F_{\lambda}(\nabla w)|^2 dx\\
&-\int_{B^+_r(-r\lambda e_n)}F_{\lambda}(\nabla v)DF_{\lambda}(\nabla v)\cdot D(v-w) dx\geq  0,
\end{split}\]
where we used the convexity of $F^2_{\lambda}$ in the last inequality. Moreover, the minimizer of $J(u)$ is unique. Assume by contradiction that there exists two different minimizers $v_1$ and $v_2$ of $J(u)$. Using again the convexity of $F^2_{\lambda}$, then we have
\[\begin{split}
J(\frac{v_1+v_2}{2})=&\frac{1}{2}\int_{B^+_r(-r\lambda e_n)}\left|F_{\lambda}\left(\frac{\nabla v_1}{2}+\frac{\nabla v_2}{2}\right)\right|^2 dx-\int_{B^+_r(-r\lambda e_n)}f^{*}\left(\frac{v_1+v_2}{2}\right) dx\\
<&\frac{1}{2}\left(\frac{1}{2}\int_{B^+_r(-r\lambda e_n)}|F_{\lambda}(\nabla v_1)|^2 dx+\frac{1}{2}\int_{B^+_r(-r\lambda e_n)}|F_{\lambda}(\nabla v_2)|^2 dx\right)dx\\
&-\frac{1}{2}\int_{B^+_r(-r\lambda e_n)}f^{*}v_1 dx-\frac{1}{2}\int_{B^+_r(-r\lambda e_n)}f^{*}v_2 dx\\
=&\frac{1}{2}J(u)+\frac{1}{2}J(v),
\end{split}\]
which is a contradiction with the fact that $J(\frac{v_1+v_2}{2})\geq J(v)$. In conclusion, we already established the uniqueness of solutions to the equation \eqref{e4.4}. In order to prove such solution is also radially symmetric, it suffices to construct an explicitly symmetric solution.

\medskip
We consider the following ODE equation:
\be\label{e5.15}
\begin{cases}
-\rho^{-(n-1)}\frac{d}{d\rho}(\rho^{n-1}v'(\rho))=f^{*},&\mbox{for any\ }\rho\in(0,r),\\
v'(0)=v(r)=0,&
\end{cases}
\ee
where $r\in\R$ is a real constant and $f^{*}$ is the capillary Schwartz symmetrization of $f$.
Integrating from $0$ to $\rho$ and using the boundary condition $v'(0)=0$, we deduce that
\be\label{e5.18}\begin{split}
v'(\rho)=&-\frac{1}{\rho^{n-1}}\int_0^{\rho} s^{n-1}f^{\#}(\kappa_{\lambda} s^n)ds\\
=&-\frac{1}{n\kappa_{\lambda}\rho^{n-1}}\int_0^{\kappa_{\lambda}\rho^n}f^{\#}(s)ds.
\end{split}\ee
Then, integrating from $\rho$ to $r$ and using the fact that $v(r)=0$, we obtain
\[v(\rho)=\int_{\rho}^r\frac{1}{n\kappa_{\lambda} t^{n-1}}\left(\int_0^{\kappa_{\lambda} t^{n}}f^{\#}(s)ds\right) dt.\]
Therefore, the solution of the problem \eqref{e5.15} is unique. Set $v(x)=v(\rho)$ with $\rho=F_{\lambda}^o(x)$.  Now, we show that $v(x)$ satisfies the equation \eqref{e4.4}. Let's denote $\frac{\partial v}{\partial x_i}$ by $v_{i}$ and $\frac{\partial F^o_{\lambda}(x)}{\partial x_i}$ by $D_i F^o_{\lambda}(x)$. It is easy to check that
\[v_{i}=v'(\rho)\frac{\partial \rho}{\partial x_i}=v'(\rho)D_iF_{\lambda}^o(x)\]
and
\[v_{ii}=v''(\rho)\left(D_iF_{\lambda}^o(x)\right)^2+v'(\rho)D_{ii}F_{\lambda}^o(x).\]
Therefore,
\[F_{\lambda}(\nabla v)=F_{\lambda}(v'(\rho)DF_{\lambda}^o(x))=v'(\rho)F_{\lambda}(D F_{\lambda}^o(x))=v'(\rho)\]
and
\[DF_{\lambda}(\nabla v)=D F_{\lambda}(v'(\rho)DF_{\lambda}^o(x)=DF_{\lambda}(DF_{\lambda}^o(x))=\frac{x}{F_{\lambda}^o(x)},\]
where we used the the property of $F_{\lambda}$ that $F_{\lambda}(DF_{\lambda}^o(x))=1$ and $F_{\lambda}^o(x)DF_{\lambda}(DF_{\lambda}^o(x))=x$.
Since
\[F_{\lambda}(\nabla v)DF_{\lambda}(\nabla v)=v'(\rho)\frac{x}{F_{\lambda}^o(x)}\]
and
\[\begin{split}
\frac{\partial}{\partial x_i}\left(v'(\rho)\frac{x_i}{F_{\lambda}^o(x)}\right)=&v''(\rho)D_iF_{\lambda}^o\frac{x_i}{F_{\lambda}^o(x)}+v'(\rho)\frac{F_{\lambda}^o(x)-x_i D_iF_{\lambda}^o(x)}{\left(F_{\lambda}^o(x)\right)^2},
\end{split}\]
then we deduce that
\[\begin{split}
\div\left(F_{\lambda}(\nabla v)DF_{\lambda}(\nabla v)\right)=&\sum_{i=1}^{N}\frac{\partial}{\partial x_i}\left(v'(\rho)\frac{x_i}{F_{\lambda}^o(x)}\right)\\
=&\frac{v''(\rho)}{F_{\lambda}^o(x)}DF_{\lambda}^o(x)\cdot x+\frac{nv'(\rho)}{F_{\lambda}^o(x)}-\frac{v'(\rho)}{\left(F_{\lambda}^o(x)\right)^2}DF_{\lambda}^o(x)\cdot x\\
=&v''(\rho)+\frac{(n-1)v'(\rho)}{F_{\lambda}^o(x)}\\
=&\frac{1}{\rho^{n-1}}\left[\rho^{n-1}v''(\rho)+(n-1)\rho^{n-2}v'(\rho)\right]\\
=&\frac{1}{\rho^{n-1}}\frac{d}{d\rho}(\rho^{n-1}v'(\rho)),
\end{split}\]
where we used the fact that $DF_{\lambda}^o(x)\cdot x=F_{\lambda}^o(x)$ in the third equation. From the preceding analysis, we obtain the identity
\[-\div\left(F_{\lambda}(\nabla v)DF_{\lambda}(\nabla v)\right)=f^{*}\quad\mbox{in\ }\WR_{\lambda,r}\]
where $\WR_{\lambda,r}:=\{x\in\R^n:F^o_{\lambda}(x)\leq r\}$ denotes the Wulff ball associated with the gauge function $F_{\lambda}$ which is coincide with the Euclidean ball $B_r(-r\lambda e_n)$.
Moreover, the boundary condition $v(r)=0$ yields
\[v(x)=0 \quad\mbox{for any\ } x\in\partial\WR_{\lambda,r}\cap\R^n_+.\]
It remains to verify the anisotropic Neumann boundary condition on $\WR_{\lambda,r}\cap\partial\R^n_+$. Observing that $Dv(x)=v'(\rho)DF_{\lambda}^o(x)$, then for any $x\in\WR_{\lambda,r}\cap\partial\R^n_+$, we derive
\[F_{\lambda}(\nabla v)DF_{\lambda}(\nabla v)\cdot\nu=v'(\rho)DF_{\lambda}(DF^o_{\lambda}(x))\cdot\nu=v'(\rho)\frac{x}{\rho}\cdot\nu=0,\]
where we used the the properties of gauge function in \eqref{e2.7}, and the last equality follows from the fact that the unit outer normal to $\partial\R^n_+$ is $\nu=e_n=(0,\cdots,1)$. We now establish the uniqueness and radial symmetry of solutions to equation \eqref{e4.4}, and then, we can use the explicit expression of $v$ to derive the conclusion. Indeed, we only need to make the variable substitution, and let $\xi=\kappa_{\lambda}t^n$, then we get
\[\begin{split}
v(\rho)=&\int_{\rho}^r\frac{1}{n\kappa_{\lambda} t^{n-1}}\left(\int_0^{\kappa_{\lambda} t^{n}}f^{\#}(s)ds\right) dt\\
=&\left(n\kappa_{\lambda}^{\frac{1}{n}}\right)^{-2}\int_{\kappa_{\lambda}\rho^N}^{\kappa_{\lambda}r^n}\xi^{\frac{2}{n}-2}\left(\int_0^{\kappa_{\lambda}\rho^n}f^{\#}(s)ds\right)d\xi.
\end{split}\]
Therefore, we can write
\be\label{e4.10}v^{\#}(s)=\left(n\kappa_{\lambda}^{\frac{1}{n}}\right)^{-2}\int_{s}^{|\Omega|}\xi^{\frac{2}{n}-2}G(\xi)d\xi.\ee
In view of \eqref{e4.9} and \eqref{e4.10}. it follows that $u^{\#}(s)\leq v^{\#}(s)$ and the proof is completed.

\medskip
Step 3. Next, we analyze the case of equality. Let $\mu$ and $\nu$ denote the distribution functions of $u$ and $v$ respectively.
 Going back to the proof of Theorem \ref{theorem4.2}, the inequalities holds as equalities for $v$ and $\nu$. If $u^{*}=v$, then $\mu=\nu$, thus
 \[1=\left(n\kappa_{\lambda}^{\frac{1}{n}}\right)^{-2}\mu(t)^{\frac{2}{n}-2}\left(-\mu'(t)\right) G(\mu(t)).\]
 This implies that we have the following equality:
\[-\frac{d}{dt}\int_{\{u>t\}}\wF_{\lambda}(\nabla u)dx=n\kappa_{\lambda}^{\frac{1}{n}}\mu(t)^{1-\frac{1}{n}}.\]
We see that the set $\{u>t\}$ achieves equality in the relative isopermetric inequality outside convex sets, thus, for any given $t>0$, the set $\{u>t\}$ is isometric to $B^+_{r_t}(-r_t\lambda e_n)$ with $|B^+_{r_t}(-r_t\lambda e_n)|=|\{u>t\}|$. Now, $\Omega=\{u>0\}$ is the increasing union of such $\{u>t\}$ for any $t>0$. Moreover, if we choose a sequence $\{t_m\}$ decreasing to zero such that each set $\{u>t_m\}$ is a ball of radius $r_m$ and centre at $-r_m\lambda e_n$ in $\R^n_+$, then $r_m$ increases to $r$ and $|\Omega|=\lim_{m\rightarrow\infty}\kappa_{\lambda} r^n_m=\kappa_{\lambda} r^n$. If $y\in\Omega$, then $y\in\{u> t_m\}$ for $m$ sufficiently large, and so $\Omega$ is isometric to $B^+_r(-r\lambda e_n)$.

\end{proof}
%it is well known that the unit normal vector of the Wulff ball is $\nu=\frac{DF_{\lambda}^o(x)}{|DF_{\lambda}^o(x)|}$, we then have
%\[F_{\lambda}(\nabla u)DF_{\lambda}(\nabla u)\cdot\nu=F_{\lambda}(\nabla u)\frac{x}{F_{\lambda}^o(x)}\cdot \frac{DF_{\lambda}^o(x)}{|DF_{\lambda}^o(x)|}=\frac{F_{\lambda}(\nabla u)}{|DF_{\lambda}^o(x)|}=\frac{v'(\rho)}{|DF_{\lambda}^o(x)|}.\]
%It follows that $F_{\lambda}(\nabla u)DF_{\lambda}(\nabla u)\cdot\nu=0$ on $B_r(-r\lambda e_n)\cap \partial\R^n_+$ if and only if $v'(\rho)=0$.

%\section{The Bossel-Daners Inequality}

Based on the above results, we obtain the capillary Bossel-Daners inequality outside convex domain.
\medskip
\begin{proof}[{\bf Proof of Theorem \ref{corollary4.1}}]
Let $u$ be the first  eigenfunction of the capillary anisotropic Laplace operator with respect to $\widetilde{\lambda}_1(\Omega;E^c)$, then $u$ satisfies
\be\label{e5.22}\begin{cases}
-\div\left(\wF_{\lambda}(\nabla u)D\wF_{\lambda}(\nabla u)\right)=\widetilde{\lambda}_1(\Omega;E^c)u&\mbox{in\ }\Omega\\
u=0&\mbox{on\ }\partial\Omega\cap E^c\\
\wF_{\lambda}(\nabla u)D\wF_{\lambda}(\nabla u)\cdot\nu=0&\mbox{on\ }\partial\Omega\cap\partial E^c.
\end{cases}\ee
We assume that $v$ solves
\be\label{e5.21}\begin{cases}
-\div\left(F_{\lambda}(\nabla v)D F_{\lambda}(\nabla v)\right)=\widetilde{\lambda}_1(\Omega;E^c) u^{*}&\mbox{in\ }B^+_r(-r\lambda e_n)\\
v=0&\mbox{on\ }\partial B_r(-r\lambda e_n)\cap \R_+^n\\
 F_{\lambda}(\nabla v)D F_{\lambda}(\nabla v)\cdot\nu=0&\mbox{on\ } B_r(-r\lambda e_n)\cap\partial \R_+^n,
\end{cases}\ee
where $u^{*}$ is the capillary Schwartz symmetrization of $u$.
Applying the inequality \eqref{e5.16}, we have
\[\int_{\Omega}u^2 dx=\int_{B^+_r(-r\lambda e_n)}(u^{*})^2dx\leq\int_{B^+_r(-r\lambda e_n)}v^2dx.\]
Moreover, by H\"{o}lder inequality, we get
\[\int_{B^+_r(-r\lambda e_n)}u^{*}vdx\leq\left(\int_{B^+_r(-r\lambda e_n)}(u^{*})^2dx\right)^{\frac{1}{2}}\left(\int_{B^+_r(-r\lambda e_n)}v^2dx\right)^{\frac{1}{2}}\leq \int_{B^+_r(-r\lambda e_n)}v^2dx.\]
Multiplying by $v$ on both sides of \eqref{e5.21} and integrating in $B^+_r(-r\lambda e_n)$, we obtain
\[\begin{split}
\widetilde{\lambda}_1(\Omega;E^c)=&\frac{\int_{B^+_r(-r\lambda e_n)}|F_{\lambda}(\nabla v)|^2 dx}{\int_{B^+_r(-r\lambda e_n)}u^{*}v dx}\\
\geq&\frac{\int_{B^+_r(-r\lambda e_n)}|F_{\lambda}(\nabla v)|^2 dx}{\int_{B^+_r(-r\lambda e_n)}v^2 dx}\\
\geq& \widetilde{\lambda}_1(B_r(-r\lambda e_n);\R^n_+),
\end{split}\]
where we used the definition of $\widetilde{\lambda}_1(B_{r}(-r\lambda e_n);\R^n_+)$ in the last inequality.

\medskip
Now, we consider the case of equality that $\widetilde{\lambda}_1(\Omega;E^c)=\widetilde{\lambda}_1(B_{r}(-r\lambda e_n);\R^n_+)$.
Let $u$ and $v$ be the solutions of \eqref{e5.22} and \eqref{e5.21} respectively.  Applying the capillary Talenti's comparison principle outside convex domain (cf. Theorem \ref{theorem4.2}), we obtain \[u^{*}\leq v\quad\mbox{in\ }B^+_r(-r\lambda e_n).\] This yields the inequality
\[-\div\left(F_{\lambda}(\nabla v)DF_{\lambda}(\nabla v)\right)\leq \widetilde{\lambda}_1v.\]
Multiplying this inequality by $v$ and integrating over $B^+_{r}(-r\lambda e_n)$, we derive
\[\int_{B^+_r(-r\lambda e_n)}|F_{\lambda}(\nabla v)|^2 dx\leq\widetilde{\lambda}_1\int_{B^+_r(-r\lambda e_n)}v^2 dx.\]
Since $\widetilde{\lambda}_1$ is also the first  capillary  eigenvalue of the anisotropic operator on $B^+_r(-r\lambda e_n)$, then we get
\[ \widetilde{\lambda}_1=\frac{\int_{B^+_r(-r\lambda e_n)}|F_{\lambda}(\nabla v)|^2 dx}{\int_{B^+_r(-r\lambda e_n)}v^2 dx}.\]
In other words, $\widetilde{\lambda}_1$ is achieved by $v$ and so $v$ is an eigenfunction of the capillary anisotropic operator, i.e., $-\div\left(F_{\lambda}(\nabla v)DF_{\lambda}(\nabla v)\right)=\widetilde{\lambda}_1v$. Thus, $v=u^{*}$ and so we are in the equality case of capillary Talenti's comparison principle outside convex doamin. Therefore, by Theorem \ref{theorem4.2}, it follows that $\Omega$ is isometric to $B^+_r(-r\lambda e_n)$.

\end{proof}

\end{document}